\documentclass[a4paper,english,cleveref]{article}
\usepackage[a4paper,
           bindingoffset=0.2in,
           left=1.0in,
           right=1.0in,
           top=1.0in,
           bottom=1.0in,
            footskip=.25in]{geometry}
\usepackage[usenames, dvipsnames]{xcolor}
\usepackage{tikz-cd}
\usepackage{tikz}
\usepackage{adjustbox}
\usepackage[normalem]{ulem}
\usepackage{multirow}
\usepackage{amsmath}
\usepackage{amsthm,thmtools}
\usepackage{hyperref}
\usepackage{float}
\usepackage[all]{xy}
\usepackage{array}
\usepackage[mathlines]{lineno}
\theoremstyle{plain} 
\newtheorem{theorem}{Theorem}[section]
\newtheorem{prop}{Proposition}[section]
\newtheorem{corollary}{Corollary}[section]
\newtheorem{conjecture}{Conjecture}[section]
\newtheorem{observation}{Observation}[section]

\newtheorem{definition}{Definition}[section]

\usepackage[linesnumbered,algoruled,boxed,lined,noend]{algorithm2e}
\usepackage{graphicx}

\usepackage{amssymb}

\def\refeq#1{\if\workingver y(\ref{#1})-[[#1]]\else(\ref{#1})\fi}
\def\refth#1{\if\workingver y\ref{#1}-[[#1]]\else\ref{#1}\fi}
\def\mylabel#1{\if\workingver y\label{#1}{\bf\ \ [[#1]]\ \ }\else\label{#1}\fi}
\def\mybibitem#1{\if\workingver y\bibitem{#1}{\bf\ \ [[#1]]\ \
}\else\bibitem{#1}\fi}

\renewcommand{\emptyset}{\varnothing}
\renewcommand{\rho}{\varrho}
\renewcommand{\phi}{\varphi}
\renewcommand{\epsilon}{\varepsilon}

\def\cV{\text{$\mathcal V$}}

\def\bE{\text{$\mathbf E$}}

\def\be{\text{$\mathbf e$}}

\newcommand{\id}{\operatorname{id}}
\newcommand{\cl}{\operatorname{cl}}

\renewcommand{\emptyset}{\varnothing}

\def\begeq#1{\begin{equation}\mylabel{#1}}
\def\endeq{\end{equation}}

\def\mathobj#1{\mbox{$#1$}}

\def\CC{\mathobj{\mathbf{C}}}

\def\FF{\mathobj{\mathbf{C}}}
\def\II{\mathobj{\mathbb{I}}}

\def\PP{\mathobj{P}}

\def\ZZ{\mathobj{\mathbb{Z}}}

\def\scalprod#1{\langle #1 \rangle}

\def\implies{\;\Rightarrow\;}

\def\setof#1{\mbox{$\{\,#1\,\}$}}

\def\0#1{\hbox{\kern25pt}$ #1 $\\}
\def\1#1{\hbox{\kern40pt}$ #1 $\\}
\def\2#1{\hbox{\kern55pt}$ #1 $\\}
\def\3#1{\hbox{\kern70pt}$ #1 $\\}

\newcounter{li}

\def\begalg#1{\begin{algo}\mylabel{#1}\normalshape:\small\baselineskip 10pt\\}
\def\endalg{\end{algo}}

\def\Figures(include=#1,cat=#2){
  \renewcommand{\textfraction}{.20}
  \renewcommand{\topfraction}{.80}
  \renewcommand{\bottomfraction}{.80}
  \renewcommand{\floatpagefraction}{.80}
  \newcount\figcount
  \figcount=0
  \let\includefigures=#1
  \def\figcat{#2}
}

\def\FigureFromFile[#1][#2](#3)#4
{
  \begin{figure}[htbp]
     \global\advance\figcount by 1
     \if\includefigures y\special{anisoscale #1.wmf, \the\hsize #2}\fi
     \vspace{#2}
     \caption{#4}
     \mylabel{#3}
   \end{figure}
}

\def\FigureFromFileTwoD[#1][#2,#3](#4)#5
{
  \begin{figure}[htbp]
     \global\advance\figcount by 1
     \if\includefigures y\special{anisoscale #1.wmf, #2 #3}\fi
     \vspace{#2}
     \caption{#5}
     \mylabel{#4}
   \end{figure}
}

\def\FigureF<#1>[#2](#3)#4
{
  \begin{figure}[htbp]
     \global\advance\figcount by 1
     \if\includefigures y\special{anisoscale \figcat/fig\number\figcount.wmf,
       \the\hsize #2}
     \fi
     \if\includefigures p
       \leavevmode
       \epsfxsize=\hsize
       \epsffile{#1}
     \fi
     \if\includefigures y
          \vspace{#2}
     \fi
     \caption{#4}
     \mylabel{#3}
   \end{figure}
}

\def\Figure[#1](#2)#3
{
  \begin{figure}[htbp]
     \global\advance\figcount by 1
     \if\includefigures y\special{anisoscale \figcat/fig\number\figcount.wmf,
       \the\hsize #1}
     \fi
     \if\includefigures p
       \leavevmode
       \epsfxsize=\hsize
       \epsffile{fig\number\figcount.eps}
     \fi
     \if\includefigures y
          \vspace{#1}
     \fi
     \caption{#3}
     \mylabel{#2}
   \end{figure}
}

\def\0{\hbox{\kern5pt}}
\def\1{\hbox{\kern20pt}}
\def\2{\hbox{\kern35pt}}
\def\3{\hbox{\kern50pt}}
\def\4{\hbox{\kern65pt}}
\def\5{\hbox{\kern80pt}}
\def\6{\hbox{\kern95pt}}
\DeclareMathOperator{\out}{out}

\newcommand{\gr}{{\sf gr}}

\newcommand{\homo}{H}

\newcommand{\mv}{\mathcal{V}}

\newcommand{\pgrad}[2]{\gr_{#1}(#2)}
\newcommand{\compl}[1]{{#1}^*}

\definecolor{yellow}{RGB}{255,225,55}
\newcommand{\tamal}[1]{\textcolor{blue}{#1}}
\newcommand\michal[1]{\textcolor{ForestGreen}{[ML: #1]}}
\newcommand\todo[1]{\textcolor{red}{[TODO: #1]}}

\newcommand{\low}{{\tt low}}

\newcommand{\leqlin}{\mathbin{\leq_{\mbox{{\scriptsize lin}}}}}

\newcommand{\Poset}{P}
\newcommand{\md}{\mathcal{M}}
\newcommand{\pers}{\mathrm{pers}}

\definecolor{dark-gray}{RGB}{64,64,64}
\definecolor{medium-gray}{RGB}{114,114,114}
\definecolor{light-gray}{RGB}{190,190,190}

\newcommand{\cancel}[1]
\title{Computing a Connection Matrix and Persistence Efficiently from a Morse Decomposition}
\author{Tamal K. Dey\thanks{Department of Computer Science, Purdue University, West Lafayette, Indiana, USA. \texttt{tamaldey@purdue.edu}}
\and Andrew Haas\thanks{Department of Computer Science, Purdue University, West Lafayette, Indiana, USA. \texttt{haas60@purdue.edu}}
\and Micha\l{} Lipi\'nski\thanks{Institute of Science and Technology, Austria.
\texttt{michal.lipinski@ist.ac.at}}
}

\begin{document}

\maketitle

\begin{abstract}
Morse decompositions partition the flows in a vector field into equivalent
structures. Given such a decomposition, one can define a further summary
of its flow structure by what is called a connection matrix.
These matrices, a generalization of Morse boundary operators from classical Morse theory, capture the connections made by the flows among the critical structures - such as attractors, repellers, and orbits - in a vector field. Recently, in the context of combinatorial dynamics, an efficient persistence-like algorithm to compute connection matrices 
has been proposed in~\cite{DLMS24}. We show that, actually, the classical persistence algorithm with exhaustive reduction retrieves connection matrices, both simplifying the
algorithm of~\cite{DLMS24} and bringing the theory of persistence closer to combinatorial dynamical systems. We supplement this main result with an observation:
the concept of persistence as defined for scalar fields naturally adapts to
Morse decompositions whose Morse sets are filtered with a Lyapunov function.
%for a Morse decomposition of a vector field derived from a connection matrix (Conley complex), and 
We conclude by presenting preliminary experimental results.

\end{abstract}
\section{Introduction}

A connection matrix is an algebraic summary of the connections between (isolated) invariant sets in a dynamical system. First proposed by R. Franzosa~\cite{Fr1986,Fr1989} for heteroclinic connections in dynamical systems, connection matrices generalize the concept of boundary homomorphims of Morse complexes in classical Morse theory~\cite{K2015}, wherein the flows between critical points of a Morse function defined on a smooth manifold are studied. A \emph{Morse complex} is a chain complex whose $i$th chain group is spanned by the critical points of index $i$ and whose boundary homomorphism is constructed by counting the number of homotopically non-equivalent paths between critical points. A connection matrix represents this boundary homomorphism.

Figure~\ref{fig:VFandCM} (left) illustrates a connection matrix under $\ZZ_2$ coefficients of a gradient flow on a sphere induced by a presumed Morse function. 
% Four critical points - repellers $A$ and $B$, saddle $C$, and attractor $D$ - give rise to a $4 \times 4$ matrix wherein an entry $X,Y \in \{A,B,C,D\}$ (indexed first by column, then by row) is non-zero if there are an odd number of homotopically unique trajectories from $X$ to $Y$ and $X$ is of one higher Morse index than $Y$. 
Four critical points - repellers $A$ and $B$, saddle $C$, and attractor $D$ - give rise to a $4 \times 4$ matrix wherein an entry indexed with $X,Y \in \{A,B,C,D\}$ (first by column, then by row) is non-zero if there are an odd number of homotopically unique trajectories from $X$ to $Y$ and $X$ is of one higher Morse index than $Y$. 
In this example, single unique (up to homotopy) trajectories exist from repellers $A$ and $B$ to $C$, whereas two trajectories exist from $C$ to~$D$. Figure~\ref{fig:VFandCM} (right) illustrates a similar connection matrix of a gradient flow on a disk with three critical points.
% Here single unique (up to homotopy) trajectories exist from repellers $A$ and $B$ to $C$, whereas two trajectories exist from $C$ to $D$.

The Morse theory for gradient vector fields has been extended by Conley~\cite{Co78}
by replacing Morse functions on smooth manifolds with flows on compact metric spaces.
In this setting critical points are replaced by isolated invariant sets, called \emph{Morse sets}, encapsulating the recurrent components of the flow.
The distinguished Morse sets constitute a \emph{Morse decomposition} (Definition~\ref{def:ms_and_md}). 
The \emph{Conley index} (Definition~\ref{def:conindex}) of a Morse set captures local dynamics in terms of
the homology group of the set relative to its \textit{exit set} which permits flow to escape.

% In the setting of Morse decomposition, instead of the Morse complex, one constructs the \emph{Conley complex} (Definition~\ref{def:conley-complex}) -- a chain complex spanned by the basis of the Conley indices of the Morse sets. 
% Similarly, the corresponding connection matrix represents the boundary homomorphism of that complex and is defined through path counting between the Morse sets.

The philosophy behind the connection matrix for a Morse decomposition is analogous to that of Morse complex.
It captures the global connections among the Morse sets in algebraic terms; it represents the boundary homomorphism of the complex formed by a direct sum of the Conley indices of the Morse sets. 
Robbin and Salamon~\cite{RoSa1992} simplified this theory, further separating algebra from dynamics through the concept of a \emph{filtered} chain complex known as the \emph{Conley complex} (Definition~\ref{def:conley-complex}).

% Appealing to the Conley index theory proposed by Conley~\cite{Co78},
% Franzosa~\cite{Fr1986,Fr1989} extended Morse theory for gradient vector fields 
% In this setting critical points are replaced by isolated invariant sets contained within the \emph{Morse sets} which constitute a \emph{Morse decomposition} (Definition~\ref{def:ms_and_md}). The \emph{Conley index} (Definition~\ref{def:conindex}) of a Morse set captures local dynamics in terms of
% the homology group of the set relative to its \textit{exit set} which permits flow to escape.

% Morse decomposition is a way to partition a flow into some equivalent classes. 
In general, a vector field does not admit a unique Morse decomposition.
In particular, a proper assembling of Morse sets may lead to a new, coarser Morse decomposition.
Moreover, there may even not exist the minimum (the most refined) Morse decomposition. 
Every such Morse decomposition can be described by a connection matrix that summarizes its essential flow structure.
% However,
% they are not unique because starting with the minimum (most refined) Morse decomposition one can build coarser Morse decompositions by merging Morse sets. Every such Morse decomposition can be described by a connection matrix that summarizes its essential
% flow structure.
\begin{figure}[htbp]
\centerline{\includegraphics[width=0.75\textwidth, keepaspectratio]{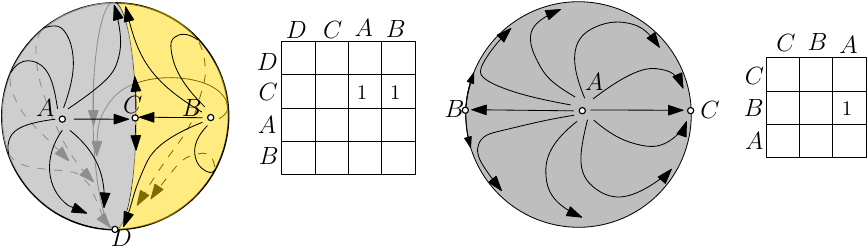}}
\caption{(left) Flow induced by a Morse function on a sphere with four critical points and its connection matrix, and (right) the same for a Morse function on a disk with three critical points.}
\label{fig:VFandCM}
\end{figure}

\begin{figure}[htbp]
\centerline{\includegraphics[width=0.75\textwidth, keepaspectratio]{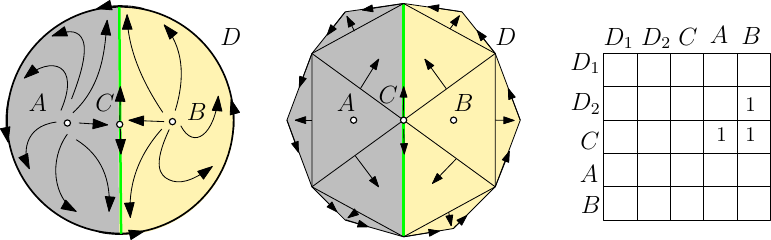}}
\caption{(left) A continuous vector field, (middle) its model with combinatorial Forman vector field, and (right) its connection matrix for the minimum Morse decomposition. The orbit $D$ is split into a saddle ($D_2$) and an attractor ($D_1$) giving a $5\times 5$ connection matrix. See an equivalent multivector field in Figure~\ref{fig:annulus}.}
\label{fig:conleycomplex}
\end{figure}

Figure~\ref{fig:conleycomplex} (left) illustrates a flow on a closed disk with repelling points $A$ and $B$, saddle point $C$, and attracting orbit $D$. 
These constitute the four Morse sets in the minimum Morse decomposition. 
However, in the connection matrix, every Morse set is represented by generators of its Conley index. 
Conley indices of the stationary points admit single generators, but the periodic orbit has two, $D_1$ and $D_2$, in dimension 0 and 1, respectively.
By definition, entries in the connection matrix indexed by generators of the same Morse set are zero, as for $D_1$ and $D_2$. 
An entry between two stationary points is determined identically as when constructing the Morse complex. 
When a more complex Morse set (for instance, a periodic orbit) is involved, the value may not be uniquely determined and depends on the representation of the Conley index generator.
The connection matrix in Figure~\ref{fig:conleycomplex} (right) is one of the two possibilities.
The other one has $1$ in the entry given by the pair $A,D_2$ instead of the pair $B,D_2$.

% The connection matrix captures the global connections among these Morse sets, and in algebraic terms represents the boundary homomorphism of the complex formed by a direct sum of the Conley indices of the Morse sets. Robbin and Salamon~\cite{RoSa1992} simplified this theory, further separating algebra from dynamics through the concept of \emph{filtered} chain complexes known as \emph{Conley complexes}.

The Conley index theory and its associated connection matrix theory provide core insights into dynamical systems defined on suitable spaces. In practice, such systems are often available only through sampled data points, necessitating the adaptation of these theories to a discrete setting where flows are defined in terms of combinatorial vectors on discrete spaces built using available data points. Forman~\cite{Forman1998a} pioneered discrete Morse theory, where vectors are defined combinatorially on CW-complexes. Figure~\ref{fig:conleycomplex} illustrates a continuous vector field (left) which is discretized and modeled by Forman vectors (middle).

Harker, Mischaikow, and Spendlove~\cite{HMS2021,HMS2021a} brought a computational framework to the theory of connection matrices. They recognized that the concept of persistence developed in topological data analysis (TDA) could be leveraged to compute connection matrices from Forman vector fields on cubical complexes. Their algorithm uses several passes to iteratively reduce an initial complex and a boundary homomorphism to a final Conley complex. In~\cite{DLMS24}, the authors discovered a stronger connection to persistence theory while giving a more
direct single-pass algorithm which operates on a generalization of Forman vector fields: \emph{multivector fields}, introduced by Mrozek~\cite{Mr2017} and later refined in~\cite{LKMW2022}. 
The connection to persistence theory was enabled by a combinatorial definition of Conley complexes 
in terms of filtered chain complexes derived from the chain spaces of simplicial complexes~\cite{MW2021b}.

Despite these advances, the algorithm {\sc ConnectMat} detailed in~\cite{DLMS24} still admits two steps which distance it from the persistence algorithm: right-to-left column additions, and more frustratingly, row additions. These steps create a larger gap than necessary between connection matrix theory and persistence theory. They also cause the algorithm run far slower than the persistence algorithm in practice (see Appendix~\ref{sec:benchmark}). In this work, we show that, given an input matrix
representing the boundary morphism of a given Morse decomposition,
{\sc ConnectMat} can be run without any right-to-left column additions or row additions to output a connection matrix. Our results imply that connection matrix theory is far closer to persistence theory than previously thought and relevant computations are almost the same as persistence computations with restriction on column additions. The result draws upon observations that link linear algebra in matrix reductions~\cite{DW22,EH2010} to Conley complex theory~\cite{MW2021b}. 

We supplement the above result with some other observations. First, we show that
the computed connection matrix essentially remains unchanged if the order of the Morse sets is shuffled while keeping the order of simplices within them intact. %\tamal{Second, we introduce a notion of persistence for vector fields as recently done in~\cite{} though in a different setting.}
Second, we observe that, similar to~\cite{BL24}, the concept of persistence for scalar fields~\cite{EH2010} can be
extended to a vector field, its Morse decompositions in particular.
For this we consider a Lyapunov function that essentially imparts an order on its Morse sets. This definition creates an avenue to compare vector fields through persistence, exactly as done for scalar fields.

%%%%%%%%%%%%%%%%%%%%%%%%%%%%%%%%%%%%%%%%%%%%%%%%%%%%%%

\section{Background and overview}
\cancel{
\subsection{Connection matrix} (to be changed)
Connection matrices are a generalization of Morse boundary operators in classical Morse theory for gradient vector fields. 
They represent the connections between Morse sets in the dynamics using homological algebra.
\begin{figure}[htbp]
\centerline{\includegraphics[width=0.75\textwidth, keepaspectratio]{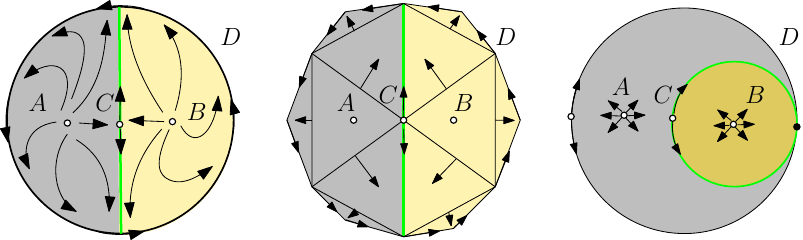}}
\caption{(left) A continuous vector field, (middle) its discretization with Forman vectors, (right) and its Conley complex represented with $5$ cells and their incidence structure.}
\label{fig:conleycomplex}
\end{figure}

The Conley index of a Morse set $M$ encapsulates the behavior of 
the local dynamics in terms of the homology
group of $M$ relative to its exit set. 
In particular, it differentiates between repellers, attractors and saddles.
Connection matrices, on the other hand, capture the dynamics 
at the global level by indicating the existence of connecting trajectories between \emph{non-trivial} Morse sets. 
In this sense, it is a useful
invariant of the complete dynamics given by the input combinatorial vector field. 
Loosely speaking,
if $H(\cl{M_p}, E_p)$ denotes the Conley index for a Morse set $M_p$, that is, the homology group of the closure of $M_p$ relative to
the exit set $E_p$, then a
connection matrix represents a boundary operator
\begin{equation*}
\Delta: \bigoplus_{p\in P} H(\cl M_p,E_p)\rightarrow \bigoplus_{p\in P} H(\cl M_p,E_p).
\end{equation*}
Alternatively,
a connection matrix can be viewed as a simplification of the vector field producing what is called a Conley complex which is a generalization of the Morse complex studied in Morse theory.
It embodies the essential dynamics by coalescing the trivial Morse sets into larger invariant
sets with a homotopy. See e.g. Figure~\ref{fig:conleycomplex} where
a vector field with two repellers, one saddle, and one attracting orbit
is simplified with five cells, two $2$-cells representing two $2$-dimensional repellers,
one $1$-cell representing the saddle, and another $1$-cell together with
a $0$-cell representing the attracting orbit. The continuous vector field shown on the left
may be discretized with a combinatorial vector field on a simplicial complex. The Conley complex
essentially simplifies this input complex to a cell complex while preserving the
Morse sets.
Depending on how the input boundary operator at the chain level of the simplicial complex
is provided for the Morse decomposition,
the homotopy-induced simplification can be different resulting in different Conley complexes.
In Figure~\ref{fig:conleycomplex}, these different Conley complexes 
indicate the fact that the system, while preserving the Morse sets, 
may be deformed by breaking the attracting periodic orbit into an attracting stationary point and a saddle. This extra saddle may be reached from only one of the two repellers. Depending on whether this is the right or the left repeller, we get two different Conley complexes.
See, for example, the two Conley complexes in Figure~\ref{fig:annulus} (right).
Corresponding to these two Conley complexes, we have two different connection matrices
shown in Figure~\ref{fig:matrixalgo}.
}

\subsection{Combinatorial multivector fields and Morse decompositions}
\label{sec:Cvec}
In this subsection, we introduce several necessary definitions from combinatorial (multi)vector field theory; see~
%\cite{DJKKLM19,DMS2020,LKMW2022,Mr2017} 
\cite{LKMW2022} for further detail.
Throughout this paper, we restrict our attention to simplicial complexes of arbitrary but finite dimension. For a simplicial complex $K$, we use $\leq$ to denote the face relation; that is, $\sigma \leq \tau$ if $\sigma$ is a face of $\tau$. We define the \emph{closure} of $\sigma$ as $\cl( \sigma ):= \{ \tau \; | \tau \leq \sigma \}$ and extend this notion to a set of simplices $A \subseteq K$ as $\cl(A) := \cup_{\sigma \in A} \cl( \sigma )$. The set $A$ is \emph{closed} if $A = \cl(A)$. 

\begin{definition}[Multivector and multivector field]
Given a finite simplicial complex $K$, a \emph{multivector} $V\subset K$ is a \emph{convex} subset, i.e., if $\sigma,\tau\in V$ with $\sigma\leq \tau$ then every simplex $\mu$ with $\sigma\leq \mu\leq\tau$ is in $V$.
%either (i) a single simplex $\sigma\in K$, or (ii) a pair of simplices
%$(\sigma,\tau)$ where $\tau$ is a coface \michal{of codimension 1} of $\sigma$. 
A \emph{multivector field} $\mv$ on $K$ is a partition of $K$ into multivectors.
\end{definition}

Following~\cite{LKMW2022}, we introduce a notion of dynamics on a combinatorial multivector field. These dynamics take the form of a multivalued map $F_{\mv} \; : \; K \multimap K$, with $F_{\mv}(\sigma) = [ \sigma ]_{\mv} \cup \cl( \sigma )$ where $[ \sigma ]_{\mv}\subset K$ 
is the unique element of the partition $\cV$ containing $\sigma$. 
% \michal{Map $F_{\mv}$ induced a directed graph $G_{\mv}$ with edges $\{(\sigma_i,\sigma_j)\in K\times K\mid \sigma_j\in F_{\mv}(\sigma_i)\}$.}
A finite sequence of simplices $\sigma_1, \sigma_2, \ldots, \sigma_n$ is a \emph{path} if for $i = 2, \ldots, n$, we have $\sigma_i \in F_{\mv{}}( \sigma_{i-1} )$.

% \michal{it's a suggestion, but I'm not fully convinced. You can remove it.}
% \begin{definition}[Morse sets and decomposition]\label{def:ms_and_md}
% A collection of mutually disjoint sets $\md=\{M_p\,|\, p\in \Poset\}$ indexed by a poset $(\Poset, \leq)$ is called a Morse \emph{Morse decomposition} of $\mv$ if 
%     \textbf{(M1)} every for every strongly connected component $C$ of $G_{\mv}$ there exists $p\in P$ such that $C\subset P$,
%     and \textbf{(M2)} for every path $\rho$ in $G_{\mv}$ from $\sigma\in M_p$ to $\sigma'\in M_q$ implies $p\leq q$.
%     We refer to elements of $\md$ as \emph{Morse sets}. 
% \end{definition}
% \michal{In particular, $K=\sqcup_{p\in \Poset}M_p$.
% }

\begin{definition}[Morse set and decomposition]\label{def:ms_and_md}
%Given a vector field $\mv$ on a finite simplicial complex $K$, a subset $M\subseteq \mv$ %\marian{$M\subseteq K$} is called
%a \emph{Morse set} if for every path $\sigma_1, \sigma_2, \ldots, \sigma_n$ with $\sigma_1, \sigma_n \in M$, each $\sigma_i$, %$i\in \{1,\ldots,n\}$ is in $M$. A collection of Morse sets
%$\md=\{M_p\,|\, p\in \Poset\}$ indexed by a poset $(\Poset, \leq)$ is called
%a \emph{Morse decomposition} of $\mv$ if we have the disjoint union
%$\mv=\sqcup_{p\in \Poset}M_p$ 
%%\marian{$K=\sqcup_{p\in \Poset}M_p$}
%and $p\leq q$ if and only if there is a path
%$\sigma_1,\ldots, \sigma_n$ with $\sigma_1\in M_p$ and $\sigma_n\in M_q$.
Given a multivector field $\mv$ on a finite simplicial complex $K$, a subset $M\subseteq K$ %\marian{$M\subseteq K$}
is called
a \emph{Morse set} if 
for every path $\sigma_1, \sigma_2, \ldots, \sigma_n$ in $K$ with $\sigma_1, \sigma_n \in M$, each $\sigma_i$, $i\in \{1,\ldots,n\}$, is necessarily in $M$. 
We call $M$ \emph{minimal} if there are no two Morse sets $M_1\neq \emptyset$
and $M_2\neq \emptyset$ so that $M=M_1\sqcup M_2$.
A collection of Morse sets
$\md_\mv=\{M_p\,|\, p\in \Poset\}$ indexed by a poset $(\Poset, \leq_P)$ is called
a \emph{Morse decomposition} of $\mv$ if 
%we have the disjoint union
$K=\sqcup_{p\in \Poset}M_p$
%\marian{$K=\sqcup_{p\in \Poset}M_p$}
and for every path
$\sigma_1,\ldots, \sigma_n$ with $\sigma_1\in M_p$ and $\sigma_n\in M_q$, we have
$q\leq_P p$. A Morse decomposition is minimum if every Morse set in the decomposition
is minimal.
\end{definition}

% \michal{We don't use this proposition anymore}
Although in this paper we consider any Morse decomposition and not necessarily
the minimum Morse decomposition, it may be interesting to know the following fact
(an easy consequence of \cite[Theorem 7.3]{LKMW2022})
which necessarily holds for minimal Morse sets.
\begin{prop}
    For any pair of simplices $\sigma$ and $\tau$ in a minimal Morse set,
    there exists a path $\sigma=\sigma_1,\ldots, \sigma_n=\tau$.
    \label{prop:minmorse}
\end{prop}
%\tamal{An astute reader will recognize that the Morse sets defined here differ slightly from the earlier definitions introduced in~\cite{DJKKLM19,DMS2020,HMS2021}\michal{are those correct papers to be cited here? I would put \cite{DMS2021, LKMW2022}}. 
%That is because we treat the trivial
%Morse sets (which are not invariant sets in earlier definitions) also as Morse sets. This
%allows us to process them in a unified way in the proposed algorithm.}

Our notion of Morse decomposition, taken from~\cite{DLMS24}, differs slightly from earlier definitions in~\cite{DJKKLM19,LKMW2022}; notably it admits \textit{trivial} Morse sets: sets $M\in\md_\mv$ containing no invariant parts.
%Definition \ref{def:ms_and_md} liberates us from a need of introducing additional concepts
%such as the invariant part $\inv\, M$ of a Morse set $M$, thereby
%simplifying the presentation of the algorithm.
%Nevertheless, the Morse decomposition in the sense of previous works can be easily retrieved by taking $\cM':=\{\inv\, M_p\mid p\in P\}$ (see e.g. \cite{LKMW2022} for the precise definition of the combinatorial invariance).
%such that $\inv\, M=\emptyset$.
%\ryan{This definition doesn't prohibit invariant sets $\rho: \mathbb{Z} \to \md$ where the image of $\rho$ is not in a Morse set. Also $\mv$ should be equal to the union of $\md$ together with the connections between them, shouldn't it? }\tamal{I deliberately used this definition to include trivial Morse sets. We need trivial Morse sets in the algorithm. So, I think we need this definition in this paper unless there is some other problem.}
%\marian{We need to tie this definition to earlier definitions in the literature, in particular to make sure that the Conley index is well defined. I understand Tamal wants to avoid defining isolated invariant sets and replace it with Morse set. This will be difficult to grasp for people in dynamics. I'd prefer to rely on one of our earlier definitions to avoid the need of adding a proof that we have a well defined Conley index.}
%\tamal{I am not sure that we should complicate things only because some people are assumed to not understand--I would argue for people in SoCG also.}
\begin{definition}[Conley index]
    The Conley index of a Morse set $M$ is defined as 
    the relative homology group $H(\cl(M),\cl(M)\setminus M)$.
    \label{def:conindex}
\end{definition}
\begin{figure}[H]
\centerline{\includegraphics[width=0.7425\textwidth, keepaspectratio]{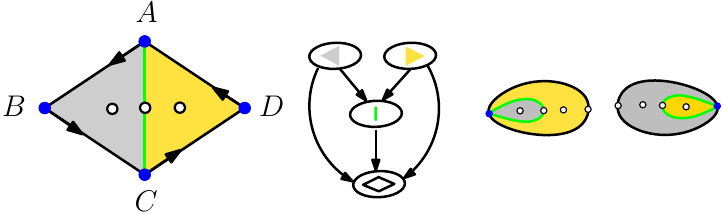}}
\caption{(left) Vector field $\mv=\{\{A,AB\},\{B,BC\},\{C,CD\},\{D,DA\},\{CA\},\{ABC\},\{CDA\}\}$; (middle) the minimum Morse decomposition consisting of 4 Morse sets, triangles $CDA$, $ABC$, edge $CA$, and
the orbit $\{\{A,AB\},\{B,BC\},\{C,CD\},\{D,DA\}\}$; and (right) two cell complexes
depicting Conley complexes - the left corresponds
to a connection matrix shown in Figure~\ref{fig:matrixalgo}.}
\label{fig:annulus}
\end{figure}

\subsection{Overview of the algorithm}
\label{subsec:algorithm_overview}

\cancel{
Let $\mv$ be a multivector field defined on a simplicial complex $K$ and
$C_*(K)$ be its chain spaces with $\ZZ_2$-coefficients. These are $\ZZ_2$-vector spaces generated by elementary chains of simplices. A grading $C_*(K)=\oplus_i C_*(M_p)$ induced by
a Morse decomposition $\sqcup M_p$ results into a filtered boundary matrix $[\partial_K]$
where $\partial_K: \oplus_iC_*(M_p)\rightarrow \oplus_iC_*(M_q)$ is the graded boundary
morphism.
This means that if $M_p$ and $M_q$ are two Morse sets with $p<q$, then simplices in the Morse set $M_p$ necessarily
come before those in the Morse set $M_q$ in the columns ordered from left to right and rows
from top to bottom.
}

Let $\mv$ be a multivector field defined on simplicial complex $K$. Let $C_*(K)$ be the collection of chain spaces of $K$ over $\ZZ_2$; these are $\ZZ_2$ vector spaces generated by elementary chains of simplices in $K$. A Morse decomposition $\md_\mv$ of $\mv$ induces a grading $C_*(K)=\oplus_p C_*(M_p)$ and thus a filtered boundary matrix $[\partial_K]$ where $\partial_K: \oplus_p C_*(M_p)\rightarrow \oplus_q C_*(M_q)$ is a graded boundary morphism. This means that in $[\partial_K]$ if $M_p, M_q \in \md_\mv$ with $p \leq_P q$, then columns and rows (ordered left-to-right and top-to-bottom respectively) representing simplices in $M_p$ come before those representing simplices in $M_q$. It is known that $[\partial_K]$ can be obtained from the minimum Morse decomposition which itself can be determined by computing strongly connected components in an appropriate graph representing $\mv$~\cite{DLMS24}.

\cancel{
A Morse decomposition can be obtained from $\mv$ by computing strongly connected components in a directed graph constructed as follows. Each simplex $\sigma\in K$ is represented as a node in the graph and a directed edge goes from a node $\sigma$ to a node $\tau$ if and only if $\tau \in F_\mv(\sigma)$. It is known that the strongly connected components in this graph constitute what are called the \emph{minimal Morse sets}~[Theorem 4.1,\cite{DJKKLM19}].
We do not elaborate
on this aspect any further as this is not the focus of this paper, and we assume
that the matrix $\partial_K$ is given.
%: C_*(K)\rightarrow C_*(K)$
%filtered according to a Morse decomposition is given.
}

The algorithm {\sc ConnectMat}, presented in~\cite{DLMS24}, computes a connection matrix through reduction of the matrix $[\partial_K]$. This algorithm bears both similarities to and key differences from the well known exhaustive persistence algorithm (see e.g.~\cite{EH2010}), which we outline below.

\textbf{Input.} The (exhaustive) persistence algorithm takes as input a filtered boundary matrix of a simplicial complex. The algorithm {\sc ConnectMat} presented in~\cite{DLMS24} also takes as input a filtered boundary matrix, however it is filtered as described above: by a poset which indexes a Morse decomposition of a multivector field on the underlying simplicial complex.

\textbf{Execution}. The persistence algorithm reduces the input matrix from left to right by way of left-to-right column additions. Every time the \emph{pivot} (row with the lowest non-zero entry) of a column (target) \emph{conflicts} with (is equal to) the pivot of a column (source) to its left, the source column is added to the target column. These additions continue until the either the target column is entirely zero, or it does not conflict with the pivot of any column to its left. Each column is considered as a target as the matrix is processed from left to right in a single pass. In exhaustive persistence, conflicts are sought and additions performed for each non-zero entry of a target column, including the pivot. The algorithm {\sc ConnectMat} performs essentially the same process as exhaustive persistence, with the following key differences: \textbf{(i)} a source column originally representing a simplex $\sigma$ triggers a conflict only if it is \textit{homogeneous}, that is, if the simplex of its pivot row belongs to the same Morse set as $\sigma$ - unlike the persistence algorithm, this disallows certain
column additions even when there are pivot conflicts; \textbf{(ii)} source columns are always to the left of target columns across Morse sets, but there is no such restriction within a Morse set; and \textbf{(iii)} the algorithm performs row additions for each column addition - if a column representing simplex $\sigma$ is added to a column representing simplex $\tau$, the row for $\tau$ is added to the row for $\sigma$.

\textbf{Output}. After processing all columns, the (exhaustive) persistence algorithm outputs a set of \textit{persistent simplex pairs} $(\sigma,\tau)$ where the pivot row of the column for $\tau$ represents $\sigma$. If the column representing a simplex $\tau$ is zero, and thus has no pivot, $\tau$ is declared \textit{unpaired}. {\sc ConnectMat} outputs a submatrix of the reduced input matrix consisting of columns (and corresponding rows) which are neither homogeneous or \textit{targetable}. A column is targetable if its corresponding row contains the pivot of a homogeneous column.

Our algorithm {\sc ConMat} (Section~\ref{sec:algorithm}) bridges many of the gaps between exhaustive persistence and {\sc ConnectMat}. The row additions of \textbf{(iii)} in {\sc ConnectMat} are actually induced by the right-to-left column additions in \textbf{(ii)}. We show that these right-to-left additions are unnecessary, eliminating both right-to-left columns additions and all row additions. The resulting algorithm {\sc ConMat} differs in execution from exhaustive persistence only in \textbf{(i)} - the necessary homogeneity of source columns.

\begin{figure}[H]
%\centerline{\includegraphics{fig/NewAlgo1step3.pdf}}
\centerline{\includegraphics{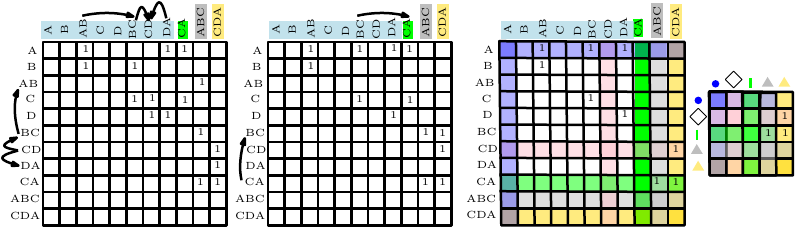}}~\\
\centerline{\includegraphics{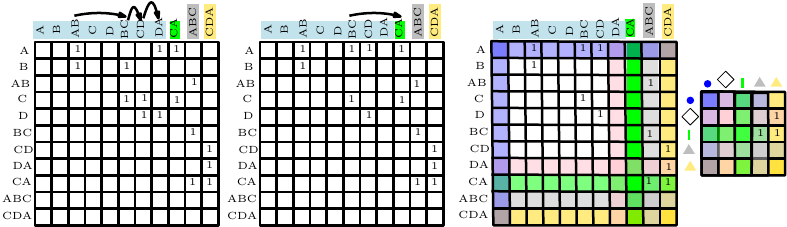}}
\caption{
Operations performed (top) by {\sc ConnectMat}~\cite{DLMS24} and (bottom) by {\sc ConMat} on a boundary matrix of the multivector field in Figure~\ref{fig:annulus}. Connection matrices are shown on the right.
}
\label{fig:matrixalgo}
\end{figure}

\subsubsection*{Example}
Figure~\ref{fig:matrixalgo} illustrates the execution of - and differences between - {\sc ConnectMat}~\cite{DLMS24} and {\sc ConMat} as both algorithms operate on a common example: the simplicial complex and associated multivector field shown in Figure~\ref{fig:annulus}. The top row of Figure~\ref{fig:matrixalgo} begins on the left with a filtered boundary matrix, uses arrows to express reductions performed by {\sc ConnectMat}, and concludes on the right with a reduced boundary matrix and connection submatrix. In the lower row, {\sc ConMat} is executed rather than {\sc ConnectMat}.

\textbf{{\sc ConnectMat} Execution.} We do not describe row additions, making throughout the assumption that if a column $i$ is added to a column $j$, the row $j$ is added to the row $i$. First, column $AB$ is added to column $BC$ due to conflict in row $B$. Next, modified column $BC$ is added to column $CD$ due to conflict in row $C$. Now, column $DA$ is added to column $CD$ due to conflict in row $D$; this addition occurs from the right, breaking the familiar pattern of exhaustive persistence. All of these additions occur within the same Morse set. Following this, column $BC$ is added to column $CA$ due to conflict in row $C$; though this addition occurs across Morse sets, it is allowed as $BC$ is a homogeneous column. Column $ABC$ triggers no conflicts, and although column $CDA$ does conflict with column $ABC$, no addition occurs as column $ABC$ is non-homogeneous.

\textbf{{\sc ConMat} Execution.} First, and entirely within the same Morse set, column $AB$ is added to column $BC$, column $BC$ to column $CD$, and column $CD$ to column $DA$. These additions are triggered by conflicts in rows $B$, $C$, and $D$ respectively. Following this, homogeneous column $BC$ is added to column $CA$.

Both algorithms yield the same connection matrix, a submatrix of the reduced boundary matrix in which homogeneous columns $AB, BC,$ and $DA$ (or $CD$ in {\sc ConnectMat}), and targetable columns $B, C,$ and $D$ are eliminated. This $5 \times 5$ matrix expresses the boundary relation on the first Conley complex shown in Figure~\ref{fig:annulus} (right), or equivalently, the following dynamics: the two ones in column $CDA$ signify flows from repeller $CDA$ to saddle edges $CA$ (extant in the initial complex), and $CD$ (formed by breaking the attracting periodic orbit); and the single one in column $ABC$ signifies flow from repeller $ABC$ toward edge $CA$.

We point out that for a given simplicial complex and Morse decomposition, there may be different boundary matrices. This occurs both because simplicies within a Morse set may be ordered differently, and because there is choice in extending the partial order on Morse sets to a total order. Different boundary matrices may lead to different connection matrices, though their associated Conley complexes depict the same dynamics up to homotopy.

\cancel{
In the example shown in Figure~\ref{fig:matrixalgo},
we illustrate the difference between the old algorithm in~\cite{DLMS24} and the
new algorithm in this paper.
In Figure~\ref{fig:matrixalgo}(left), we show the boundary operator 
filtered according to the Morse decomposition shown
in Figure~\ref{fig:annulus}(left). 
First, we describe the working of the algorithm in~\cite{DLMS24} on the matrix.
It is illustrated in the top row. We do not elaborate on the row operations with the implicit understanding that whenever a column $i$ is added to column $j$, the row $j$ is added to row $i$.
The column $AB$ is added to column $BC$ due to the conflict in the row $B$
even though it is not the pivot row
of column $BC$. 
Next, the modified column $BC$ from left is added to $CD$ because of a conflict
in row $C$.
These additions are akin to additions in 
the persistence algorithm executed with exhaustive reductions~\cite{EH2010}. 
In the old algorithm, apart from exhaustive reductions from left, columns are also
added from right. So, the column $DA$ from right is added to column $CD$ due
to a conflict in row $D$ which makes $CD$ a zero column.
All these additions happen within the same Morse set. In the meantime, corresponding row additions affect columns $ABC$ and $CDA$. Next addition happens between
columns $BC$ and $CA$ which belong to different Morse sets, but the addition of $BC$ to $CA$ is allowed
because $BC$ is homogeneous (simplex $C$ belongs to the same Morse
set as $BC$). This column addition triggers a row addition of $CA$ to $BC$. The column $ABC$ does not trigger any conflict. The final column $CDA$ has
a conflicting pivot row with column $ABC$. However, $ABC$ is not added to $CDA$
because it is not homogeneous (simplices $CA$ and $ABC$ does not come from the same Morse set). 

Next, we describe the working of the new algorithm which only employs exhaustive reductions
from left. It is illustrated in the bottom row. First, column $AB$ is added to column $BC$
and the resulting column $BC$ is added to column $CD$ as before. Now $CD$ is added to $DA$
due to the conflicting row $D$ making $DA$ a zero column. 
These additions happen within the same Morse set without
any additions from right and any additions in rows. The last addition happens from 
$BC$ to $CA$ which makes $CA$ a zero column. 
%After this, no more additions are required. 

One main departure of both algorithms from the classical persistence algorithm is the way the output
is determined which is not the entire reduced matrix but only a submatrix. 
%A column in the reduced
%matrix is called \emph{targetable} if its simplex appears as a lowest entry in a homogeneous column. 
%as they appear as the lowest entry in 
All homogeneous and targetable 
columns and their corresponding rows are eliminated from the output. 
For example, the homogeneous columns $AB$, $BC$, and $DA$ ($CD$ in the old algorithm) and
the targetable columns $B$, $C$, $D$ are eliminated from the output matrix.

The resulting
$5\times 5$ connection matrices are shown which represents the dynamics depicted by the
first Conley complex shown in Figure~\ref{fig:annulus} (right). Both algorithms produce
the same connection matrix.
The column for $CDA$ in the connection matrix
has two $1$'s in rows $CA$ and $CD$ signifying that the dynamics flow from the repeller triangle
$CDA$ to 
the existing saddle edge $CA$ and the new saddle edge $CD$ arising 
from breaking the attracting periodic orbit. 
Similarly, the column $ABC$ has a single $1$ in row $CA$ 
but no $1$ in row $CD$ indicating that dynamics 
can flow from the triangle $ABC$ only to the repeller edge $CA$.

We point out that for a given Morse decomposition, the boundary matrix
can be different because the simplices in the same Morse set can be filterted
differently. Such a different filtered boundary matrix for the same
vector field may give a different connection matrix though their Conley complexes
depict the same dynamics up to a homotopy. 
%In~\cite{DLMS24}, authors
%show a different filtration of the boundary matrix for the vector field
%in Figure~\ref{fig:annulus} (left) and arrive at a different connection matrix with the old algorithm. 
%One can verify that our new algorithm also arrives at the same connection matrix.
%The second Conley complex shown in Figure~\ref{fig:annulus} corresponds to this connection matrix.
}

\section{Algebraic formulation}
\label{sec:prelim}

Although connection matrices were constructed by R.\ Franzosa in \cite{Fr1986,Fr1989}  
as a tool facilitating the detection of heteroclinic connections between invariant sets of dynamical systems,
they may be decoupled from dynamics, defined purely algebraically and studied as a part of homological algebra. In this formal setting, a simplicial complex $K$ is presented alongside a multivector field $\mv$ which partitions $K$ into convex sets, and a Morse decomposition. This decomposition is viewed as a 
\emph{chain complex} with a \emph{boundary map} that is \emph{filtered} w.r.t. a poset $\mathbb{P}=(P,\leq_P)$ over which the Morse decomposition is supported. This input
chain complex is iteratively converted to other chain complexes, the last
being the source of a Conley complex. 

\cancel{
Successive chain complexes are always kept connected
with isomorphic \emph{chain maps} presented as matrices w.r.t. some appropriate bases.
The Conley complex is extracted from the last chain complex by an implicit
chain homotopy that is implemented by a simple removal of a certain submatrix 
of the matrix representing the final chain map. 
%***The connection to homology and and hence
%Conley index of the Conley complex is obtained by
%applying the homology functor on the
%\emph{filtered chain complexes} which gives
%a \emph{graded vector space} with \emph{filtered
%linear maps***. How about this? }
Using homology groups with a field coefficient for each individual simplicial complex in the
\emph{filtered chain complexes}, we obtain \emph{graded vector spaces} connected by 
\emph{filtered linear
maps} induced by chain maps.
Our algebraic framework 
builds on these concepts/objects, which have
been used in earlier works in the context of connection matrix~\cite{HMS2021a,MW2021b}.

Before presenting this formal, purely algebraic definition of connection matrix,
we present some general assumptions, notations, and
%fix some notation and recall 
some basic facts from homological algebra. 
}
%%%%%%%%%%%%%%%%%%%%%%%%%%%%%%%%%%%%%%%%%%%%%%%%%%%%%%
\cancel{
\subsection{Algebraic preliminaries}
Throughout the paper we consider only finite dimensional vector spaces over the field $\ZZ_2$ and matrices with $\ZZ_2$ entries.
Let $V$ be an $n$-dimensional vector space over  $\ZZ_2$.
Consider the ordered set 
$\II_n:=(1,2,\ldots,n)$.
Since in the paper the order of vectors in a basis $B=(b_1,b_2,\ldots,b_n)$ of $V$ matters,
we consider the basis as a map $b:\II_n\ni i\mapsto b_i\in V$.
Given a basis $B$, we denote the associated scalar product by $\scalprod{\cdot,\cdot}$.
More precisely,  
$\scalprod{\cdot,\cdot}$ is
the bilinear form $V\times V\to\ZZ_2$ defined on generators $b,b'\in B$ by
\[
\scalprod{b,b'}=\begin{cases}
                 1 & \text{ if $b=b'$,}\\
                  0 & \text{ otherwise}
                \end{cases}
\]
and linearly extended to $V\times V$.
An example of an $n$-dimensional vector space over $\ZZ_2$ is the coordinate space $\ZZ_2^n$.
The {\em canonical basis} $\bE_n=\{\be_1,\be_2,\ldots,\be_n\}$ consists of vectors $\be_i$ in which all coordinates except the $i$th are zero.

Given a matrix $A$, we denote its $i$th row by $A[i,\cdot]$, its $j$th column by $A[\cdot,j]$
and its entry in $i$th row and $j$th column by $A[i,j]$.
By $\low_A(j)$ we denote the row index of the lowest $1$ in $A[\cdot,j]$ if the column is nonzero and we set $\low_A(j)=0$ otherwise. 
%Similarly, by $\first_A(i)$ we denote the column index of the first $1$ in $A[i,\cdot]$ if the row is nonzero and we set $\first_A(i)=\infty$ otherwise.
If $i:=\low_A(j)$ is nonzero,  we say that column  $A[\cdot,i]$ is the {\em target} of column  $A[\cdot,j]$. 

%\begin{comment}
%Assuming basis $B$ is fixed, a square $n\times n$ matrix $A=(a_{ij})$ with entries in $\ZZ_2$ induces a linear map $h_A:V\to V$ defined on  basis $B$
%by $h_A(b_i):=\sum_{i=1}^{n}a_{ij}b_j$. 
%\end{comment}

Let $V'$ be an $n'$-dimensional vector space over  $\ZZ_2$ with basis $B'=\{b'_1,b'_2,\ldots,b'_{n'}\}$.
Given a linear map $h:V\to V'$, its {\em matrix in bases} $B$ {\em and} $B'$ is the matrix $(h_{ij})$ where
\begin{equation}
\label{eq:hij}
h_{ij}:=\scalprod{hb_j,b'_i} \text{ for $i\in\II_{n'}$ and $j\in\II_n$.}
\end{equation}
Conversely, given a $\ZZ_2$-matrix $A=(a_{ij})$ with $n'$ rows and $n$ columns 
we have a linear map $h_{B,A,B'}:V\to V'$ defined on  basis $B=(b_1,b_2,\ldots,b_n)$
by $h_{B,A,B'}(b_j):=\sum_{i=1}^{n'}a_{ij}b'_i$. 
Then, clearly, $A$ is the matrix of $h_{B,A,B'}$.
If bases $B$ and $B'$ are clear from context, we adopt the simplified notation $h_A$.

Given a fixed basis $B$ of $V$, we identify $V$ with the coordinate space $\ZZ_2^n$ via isomorphism $h_{B,I_n,\bE_n}$
where $I_n$ denotes an $n\times n$ identity matrix. 
Under this identification, composition of linear maps corresponds to multiplication of matrices.

We recall (see, for instance, \cite[Section 3.1]{KaMiMr2004}) 
that for $1\leq i<j\leq n$ the operation of adding $i$th column to $j$th column followed by 
adding $j$th row to $i$th row in an $n\times n$ $\ZZ_2$-matrix $A$
results in a new matrix $A'$ such that 
\begin{equation}
\label{eq:AE}
E_{i,j} A'=A E_{i,j}
\end{equation}
where $E_{i,j}$ (illustrated below) is the sum of the identity matrix and the matrix whose all entries are zero except a one in $i$th row and $j$th column.
We call it the {\em column/row addition} matrix.
%\begin{comment}
      \begin{equation*}
      \label{eq:cl-matrix}
        E_{i,j}:=\begin{array}{cc}
                     \left[
                     \begin{array}{ccccccccc}
                       1 & & & & & & & &  \\
                         &.& & & & & & &  \\
                         & &1&0&.&0&1& &  \\
                         & & &1& & &0& &  \\
                         & & & &.& &.& &  \\
                         & & & & &1&0& &  \\
                         & & & & & &1& &  \\
                         & & & & & & &.&  \\
                         & & & & & & & & 1\\
                     \end{array}
                     \right]
                     &
                     \begin{array}{c}
                          \\
                          \\
                         (i)\\
                          \\
                          \\
                          \\
                         (j)\\
                          \\
                          \\
                     \end{array}
                   \end{array}
      \end{equation*}
 %     is the {\em column/row addition} matrix.
%\end{comment}
}
We briefly explain this conversion process of chain complexes in terms of algebra, leading to an algebraic view of
Conley complexes. For further details, see~\cite{MW2021b}.

\paragraph{Graded vector spaces and linear maps} Let $P$ be a finite set.
By a $P$-{\em gradation} of a finite dimensional $\ZZ_2$-vector space $V$ we mean 
the collection $\setof{V_p\mid p\in P}$ of subspaces of $V$ such that $V=\bigoplus_{p\in P}V_p$.
We call such a vector space with a given $P$-gradation a $P$-{\em graded vector space}.
Given a $P$-graded vector space $V=\bigoplus_{p\in P}V_p$ 
we denote by $\iota^V_{q}:V_q\to V$ and $\pi^V_{p}:V\to V_p$ respectively the inclusion and projection homomorphisms.
Given another $P$-graded vector space $V'=\bigoplus_{p\in P}V'_p$ 
we identify a linear map $h:V\to V'$ with the matrix $[h_{pq}]_{p,q\in P}$ of partial linear maps
$h_{pq}:V_q\to V'_p$ where $h_{pq}:= \pi^{V'}_p\circ h\circ \iota^V_q$.

\paragraph{Chain complexes and homotopy}
We recall that a {\em chain complex} with $\ZZ_2$ coefficients is a pair $(C,d)$ where $C=\bigoplus_{q\in\ZZ}C_q$
is a $\ZZ_2$ vector space with gradation $C=(C_q)_{q\in \mathbb{Z}}$ 
and $d:C\to C$ is a linear map satisfying $d(C_q)\subset C_{q-1}$ and $d^2=0$.
In our context, $C_q$ will be the chain space consisting of $q$-chains
$\sum_{\tau\in K^q}\alpha_{\tau}\tau$, $\alpha_\tau\in \ZZ_2$,
where $K^q$ denotes the set of $q$-simplices
in a simplicial complex $K$ that supports the vector field
in consideration.

Given another such chain complex $(C',d')$, a {\em chain map} $\varphi:(C,d)\to(C',d')$ is a linear map such that 
$\varphi(C_q)\subset C'_q$ and $d'\varphi=\varphi d$.
A chain map is a {\em chain isomorphism} if it is an isomorphism as a linear map. 
Two chain maps $\varphi,\psi:(C,d)\to(C',d')$ are {\em chain homotopic} if there exists a linear map $S: C\to C'$
satisfying $S(C_q)\subset C'_{q+1}$ and $\psi-\varphi=d' S+S d$. Such an $S$ is called a {\em chain homotopy}.
Two chain complexes $(C,d)$, $(C',d')$ are {\em chain homotopic} if there exist chain maps $\varphi:(C,d)\to(C',d')$
and $\psi:(C',d')\to(C,d)$ such that $\psi\varphi$ is chain homotopic to $\id_C$ and $\varphi\psi$ is chain homotopic to $\id_{C'}$.
In what follows, we consider only finite dimensional chain complexes, allowing
us to work with finite bases.

\paragraph{Filtered complexes and maps} Consider now a fixed finite poset $(P,\leq_P)$ and a linear map $h:V\to V'$ where $V$ and $V'$ are $P$-graded vector spaces. 
We say that $h$ is $\PP$-{\em filtered} or briefly {\em filtered} when $\PP$ is clear from the context if
\begin{equation*}
\label{eq:filt-homo}
h_{pq}\neq 0 \implies p\leq_P q.
\end{equation*}
\cancel{
The following proposition is straightforward.
\begin{prop}
\label{prop:extension}
If $h$ is $\PP$-filtered then it is $\PP'$-filtered for any poset $\PP'=(P,\leq')$ with partial order $\leq'$ extending $\leq$.
\end{prop}
}
By a $\PP$-{\em filtered} chain complex we mean
a chain complex $(C,d)$ with field coefficients and a given gradation $C=\bigoplus_{p\in P}C_p$
such that the boundary homomorphism $d$ is $\PP$-{\em filtered}.
Given another $\PP$-filtered chain complex  $(C',d')$ we define a $\PP$-{\em filtered chain map} 
$\varphi:(C,d)\to (C',d')$ as a chain map which is also $\PP$-filtered as a homomorphism.
One can easily check that a composition of filtered chain maps is a filtered chain map. 
We say that $\varphi$ is a {\em filtered chain isomorphism} if it is a filtered chain map which is also an isomorphism. 
Note that trivially the identity homomorphism on $(C,d)$, denoted $\id_C$, is a filtered chain isomorphism.
Two filtered chain complexes $(C,d)$ and  $(C',d')$  are {\em filtered chain isomorphic} if there exist filtered chain maps
$\varphi:(C,d)\to (C',d')$ and $\varphi':(C',d')\to (C,d)$ such that $\varphi'\circ\varphi=\id_C$ 
and $\varphi\circ\varphi'=\id_{C'}$.

\paragraph{Conley complexes and connection matrices}
\label{sec:concomplex}
Two filtered chain maps $\varphi,\varphi':(C,d)\to (C',d')$ are said to be {\em filtered chain homotopic} if there exists
a chain homotopy joining  $\varphi$ with $\varphi'$ which is also filtered as a homomorphism.
Two filtered chain complexes $(C,d)$ and  $(C',d')$  are {\em filtered chain homotopic} if there exist filtered chain maps
$\varphi:(C,d)\to (C',d')$ and $\varphi':(C',d')\to (C,d)$ such that $\varphi'\circ\varphi$ is filtered chain homotopic to $\id_C$
and $\varphi\circ\varphi'$ is filtered chain homotopic to $\id_{C'}$.

A {\em Conley complex} of a filtered chain complex $(C,d)$ is defined as 
any filtered chain complex $(\bar{C},\bar{d})$
which is filtered chain homotopic to $(C,d)$ and satisfies $\bar{d}_{pp}=0$ for all $p\in\PP$;
see~\cite{HMS2021,MW2021b}.
The following theorem may be easily obtained as a consequence of results in~\cite[Theorem 8.1, Corollary 8.2]{RoSa1992} and~\cite[Proposition 4.27]{HMS2021}.
\begin{theorem}
\label{thm:connection-matrix}
For every finitely generated chain complex $(C,d)$ there exists a Conley complex and any two Conley complexes of $(C,d)$ are filtered chain isomorphic. 
\qed
\end{theorem}

Theorem~\ref{thm:connection-matrix} lets us define a {\em connection matrix} of $(C,d)$ 
as the matrix of homomorphisms $[\bar{d}_{pq}]_{p,q\in\PP}$ for any Conley complex  $(\bar{C},\bar{d})$ of  $(C,d)$. A Morse decomposition $\md_\mv=\{M_p\,|\, p\in \Poset\}$ of a vector
field $\mv$ defines a $P$-filtered 
chain complex $(C,d):=(C_{\md_\mv},d_{\md_\mv})$ where each vector space $C_p$, $p\in P$, is formed by
chains of simplices in $M_p$ and where $d$ is given by the boundary morphism among these chain spaces. This allows us to define:
\begin{definition}[Conley complex for Morse decomposition]\label{def:conley-complex}
  For a Morse decomposition $\md_\mv$, any $(\bar{C}_{\md_\mv},\bar{d}_{\md_\mv})$ is called its Conley complex and $[\bar{d}_{\md_\mv}]$ its
  connection matrix.  
\end{definition}
%We note that each homomorphism in the connection matrix is represented itself as a matrix.
%Moreover, this matrix is often $1\times 1$ matrix, consisting of just one number, which is the case in our setting.  

\paragraph{Matrices for boundary homomorphisms}
%\michal{here is no consistency here, sometimes we have $P$ sometimes $\mathbb{P}$.
%Should we put $P$ everywhere?}
%A Morse decomposition $\md_\mv$ can be computed from $\mv$ by computing
%strongly connected components in an appropriate directed graph~\cite{DLMS24} and
%coarsening them.
Given a $P$-filtered chain complex $(C,d)$, we represent $d$ 
with a matrix in some chosen basis. When $(C,d)=(C_{\md_\mv},d_{\md_\mv})$,
a particular basis of interest
for $C$ is given by the elementary chains $\{\langle\sigma\rangle\}_{\sigma\in K}$, $\langle \sigma \rangle=\sum_{\tau\in K}\alpha_{\tau}\tau$ where
$\alpha_\tau$ is $1$ if $\tau=\sigma$ and is $0$ otherwise.

For a $P$-graded vector space $V$, we say that basis $B=(b_1,b_2,\ldots,b_n)$ is $P$-{\em graded}
if $B\subset\bigcup_{p\in P}V_p$. 
For each basis vector $b_i\in B$ there is exactly one $p\in P$ such that $b_i\in V_p$, and
we denote this $p$ by $\pgrad{P}{b_i}$.
Clearly, every $P$-{\em graded vector space}, and hence the chain complex $C$, admits a $P$-graded basis. 
Indeed, the elementary chains
$\{\langle \sigma \rangle\}_{\sigma\in K}$ constitute a $P$-graded basis
for $C_{\md_\mv}$ with $\pgrad{P}{\langle \sigma \rangle}=p$ if $\sigma\in M_p$.

Given a fixed $P$-graded basis $B=(b_1,\ldots,b_n)$ of $V$, we say that
$B$ is $\PP$-filtered if there is a linear extension $\leqlin$ of $\leq_P$ such that
\begin{eqnarray*}
i\leq j\implies \pgrad{P}{b_i}\leqlin \pgrad{P}{b_j}.
\end{eqnarray*}
An ${n}\times n$ matrix $A$ is presented w.r.t. a basis $B=(b_1,\ldots,b_n)$ if its
$i$th column and row represent the basis element $b_i$.
We say $A$ is $\PP$-{\em filtered} w.r.t. $B$ or briefly {\em filtered} (when $\PP$ and $B$ are clear from the context) if $B$ is $\PP$-filtered. 

Given a $\PP$-filtered chain complex $(C,d)$, a $\PP$-filtered boundary
matrix $A=[d]$ is obtained by considering a $\PP$-filtered chain basis of $C$
where the map $h_A:C\to C$ corresponding to the matrix $A$ is $d$.
Such a matrix always exists and is upper triangular~\cite{MW2021b}. 

\cancel{
\begin{prop}
\label{prop:dab-exists}
A $P$-filtered Chain complex $(C,d)$ admits a filtered boundary matrix $A$ such that the linear map
$h_A: C\rightarrow C$ is $P$-filtered.
\end{prop}
}

Indeed, for a $\PP$-filtered chain complex $(C_{\md_\mv},d_{\md_\mv})$ obtained from a Morse decomposition $\md_\mv=\{M_p\,|\, p\in \Poset\}$, 
the $\PP$-filtered matrix $A$ w.r.t. a $\PP$-filtered elementary chain basis
$(\langle \sigma_1\rangle,\ldots,\langle\sigma_n\rangle)$ gives
$h_A=d$ if $A[i,j]=\alpha_{i,j}$ where $\partial \sigma_j=\sum_{i}\alpha_{i,j}\sigma_i$ for all $i,j\in \{1,\ldots, n\}$.

%%%%%%%%%%%%%%%%%%%%%%%%%%%%%%%%%%%%%%%%%%%%%%%%%%%%%%
\section{Algorithm}
\label{sec:algorithm}

We present our algorithm {\sc ConMat}. For comparison, we also present the older algorithm {\sc ConnectMat} of \cite{DLMS24}. Given a filtered chain complex $(C,d):=(C_{\md_\mv},d_{\md_\mv})$, these algorithms compute a connection matrix for $\md_\mv$
from its filtered boundary matrix $A$, which is given with respect to a $\PP$-filtered basis comprised of elementary chains $(\langle \sigma_1\rangle,\ldots,\langle\sigma_n\rangle)$. Definition~\ref{def:alg_notations} formalizes several necessary concepts.

\begin{definition}\label{def:alg_notations}
Let $j$ be any column in $A$. The notation
$\low_A(j)$ denotes
the largest row index $i$ in $A$ so that $A[i,j]$ is non-zero. We also
refer to $\low_A(j)$ as the \emph{pivot} row of $j$.
We say column $j$ and row $j$ are \emph{homogeneous} if $\sigma_j$ and $\sigma_{\low_A(j)}$
are in the same Morse set. If column $j$ is homogeneous, then the column and row
with index $\low_A(j)$ are called \emph{targetable}. The set of indices of
homogeneous and targetable columns (and hence rows)
in matrix $A$ are denoted $J_h$ and
$J_t$ respectively. Let $\II_n$ denote the set $\{1,\ldots,n\}$.
\end{definition}

\begin{algorithm}[H]
    \caption{{\sc{ConnectMat}}}\label{alg:connectmat}
    \KwData{An $n\times n$ matrix $A$ of a filtered boundary homomorphism $d$} 
    \KwResult{A connection matrix}
    \For(){$j:=1$ to $n$}{ %\kwdo\\
%\>\kwwhile $\exists (\tts < \ttj)$ \tts.t. $\tts$ is reducible and $A(\low(\tts),\ttj)=1$\\
        \For{$i:=\low_{A}(j)$ \text{down to} $1$}{
%             $S:=\emptyset$ \;
            \If{$A[i,j]=1$}{ 
                $S:=\setof{s\in\II_n\mid \text{$s\neq j$ \& $\low_{A}(s)=i$ \& $A[\cdot,s]$ is homogeneous}}$\;
                \If{$S\neq\emptyset$}{ 
                    $s:=\min S$\;
                    add column $A[\cdot,s]$ to column $A[\cdot,j]$\;
                    add row $A[j,\cdot]$ to row $A[s,\cdot]$\; 
                }
            }
        }
%\>\kwendwhile\\
    }
$J:=\II_n\setminus J_{h}(A)\setminus J_{t}(A)$\;
\Return{$A$ restricted to columns and rows with indices in $J$}
\end{algorithm}

\begin{algorithm}[H]
    \caption{{\sc{ConMat}}}\label{alg:contmat}
    \KwData{An $n\times n$ matrix $A$ of a filtered boundary homomorphism $d$} 
    \KwResult{A connection matrix}
    \For(){$j:=1$ to $n$}{ %\kwdo\\
%\>\kwwhile $\exists (\tts < \ttj)$ \tts.t. $\tts$ is reducible and $A(\low(\tts),\ttj)=1$\\
        \For{$i:=\low_{A}(j)$ \text{down to} $1$}{
%             $S:=\emptyset$ \;
            \If{$A[i,j]=1$}{ 
                \If {$\exists$ \text{$s<j$ \& $\low_{A}(s)=i$ \& $A[\cdot,s]$ is homogeneous}}
                   {add column $A[\cdot,s]$ to column $A[\cdot,j]$
                }
            }
        }
    }
%\>\kwendwhile\\  
$J:=\II_n\setminus J_{h}(A)\setminus J_{t}(A)$\;
\Return{$A$ restricted to columns and rows with indices in $J$}
\end{algorithm}

Both algorithms consist of a reduction phase followed by a simple extraction phase. The reduction phase in {\sc ConnectMat} consists left-to-right and right-to-left additions of source columns to target columns, as well as row additions. The reduction phase of {\sc ConMat} consists only of left-to-right additions of source to target columns.

\subsection{Reductions}
\label{sec:reductions}

 In both algorithms, let $A_{out}$ denote the matrix $A$ after all reductions are performed, that is, $A_{out}$ is the transformed matrix $A$ when the algorithms exit the outermost {\bf for} loop.

We say that a filtered matrix $A$ is {\em reduced} if it satisfies the following three conditions:
\begin{itemize}
   \item (R1) the map
$
  \alpha: J_h(A)\ni j\mapsto \low_{A}(j)\in J_t(A)
$
is a well defined bijection, 
   \item (R2) $J_h(A)\cap J_t(A)=\emptyset$, and
   \item (R3) if $A[i,j]=1$ for some $i\leq \low_A(j)$, then there is no $s\in J_h(A)$ with $s<j$ and $\low_A(s)=i$.
\end{itemize}

In~\cite{DLMS24}, it is proved that the matrix $A_{out}$ after all reduction steps in {\sc ConnectMat} is reduced, and that the submatrix of $A_{out}$ obtained by eliminating all homogeneous and targetable columns and rows
is a connection matrix for the input chain complex $(C,d)$. It is easy to verify that {\sc ConMat} satisfies R1 and R3 due to additions from homogeneous columns within
the innermost \textbf{for} loop. The difficult part is to show that it satisfies R2 also. 

\cancel{
\begin{itemize}
    \item Just to remind myself about why the older algorithm works with row additions (there is a gap in the proof in the current journal version, this note says that the algorithm is alright): Older algorithm adds a homogeneous column $d$ to a column $c$ on its left. This will need to add row($c$) to row($d$). We show that this row additions do not cause any change to the columns left of $d$ and thus $\low(d')$ for any column does not get greater than $\gr(d')$.
    The only columns that can get into this trouble are the ones between $c$ and $d$ because the matrix is always (strictly) upper triangular. We observe that
    the row($c$) have all zeros between the columns of $c$ and $d$. Take any column $d'$ that has a $1$ in row($c$). This means $d'$ represents
    a $(p+1)$-simplex if $c$ represents a $p$-simplex. Since $d$ has been added
    to $c$, the column $d$ also represents a $p$-simplex. We claim that $d$ and $c$ are in the same Morse set. Since $d$ is homogeneous, $\low(d)=\low(c)$ is in the same Morse set as of $d$. Since column $c$ lies between the column $\low(c)$ and $d$, the column $c$ has to be in the same Morse set as  $d$ is in. We can sort the columns of all Morse sets in any order as long as all faces appear before a simplex appear in the Morse set. So, we can sort all simplices by their dimensions in any Morse set. In particular, the $(p+1)$-simplices are placed after $p$-simplices. This implies that row($c$) cannot have any $1$ between the columns of $c$ and $d$.
\end{itemize}
}

Before we prove this fact, we establish certain notations. 
In what follows, we consider a \emph{complete} reduction of a matrix which, for example, can be achieved by 
left-to-right column additions while applying no restriction to homogeneity. 
%The reduction in algorithm {\sc ConMat} becomes a complete reduction
%if the condition ``$A[\cdot,s]$ is homogeneous'' in step 4 is omitted.
\begin{definition}
Let $\compl A$ be obtained from a matrix $A$ after a set of
left-to-right column additions henceforth called a \emph{complete column
reduction} where every column $j$ in $\compl A$ is either zeroed out or $\low_A(j)\neq \low_A(j')$ for $j\neq j'$. We say $\compl A$ is a \emph{complete
reduction} of $A$.
\end{definition}
\begin{definition}
We say two columns $s_1$ and $s_2$
are in the same Morse set if $\sigma_{s_1},\sigma_{s_2}$ belong to the same Morse
set and write $s_1\simeq s_2$.
\end{definition}

\begin{observation}\label{obs:non-homogeneous-columns}
    Throughout both {\sc ConMat} and a complete column reduction, a non-homogeneous column cannot become homogeneous.
\end{observation}

% \begin{prop}
%     Both in {\sc ConMat} and in a complete reduction,
%     the pivot and the homogeneity of a column $j$ in matrix $A$ can change only by the addition of a homogeneous column $s$ such that $s\simeq j$.
% \end{prop}
\begin{prop}\label{prop:homogeneity_loss}
    Consider matrix $A$ (at any stage of {\sc ConMat} or a complete column reduction) and a homogeneous column $j$.
    Let $s$ be another column such that $s<j$ and $A[\low_A(s),j]=1$.
    If an addition of a column $s$ to $j$ changes the pivot of $j$ 
    % makes $j$ non-homogeneous 
    then $s$ is a homogeneous column such that $s\simeq j$ and $\low_A(s)=\low_A(j)$.
\end{prop}
\begin{proof}
    The addition can change the pivot of $j$ only if $t:=\low_A(s)=\low_A(j)$.
    Since $j$ is homogeneous we have $t\simeq j$.
    Therefore, $t < s <j$, which implies that $s\simeq j$.
    \qed
\end{proof}

%\begin{prop}
%Let $\bar A$ be a completely reduced matrix of $A$. The value $\low_{\bar A}(j)$ remains the same no matter which left-to-right column
 % additions are used to reduce $A$ completely.
 % \label{prop:complete-reduction}
%\end{prop}
\cancel{
\begin{prop}\label{prop:hom_coincides_with_complete_reduction}
    Consider a matrix $A_{\out}$ obtained from {\sc ConMat} and its complete reduction 
    $\compl{A}$.
    A column $j$ is homogeneous in $A_{\out}$ if and only if $j$ is homogeneous in $\compl{A}$.
\end{prop}
\begin{proof}
    Non-homogeneous columns of $A$ do not change their homogeneity status by Observation~\ref{obs:non-homogeneous-columns}. They are also not used to
    reduce a homogeneous column in {\sc ConMat} and 
    in a complete reduction. Therefore, we can forget about these columns in $A$ and
    without loss of generality assume that $A$ consists of only homogeneous columns.

    Suppose that, a homogeneous column $j$ has become non-homogeneous in $A_{\out}$.
    We prove that it becomes non-homogeneous in $\bar A$ as well. 
    %Consider the Morse set $M$ that $j$ belongs to and the submatrix of $A$
    %that consists of only columns in $M$ and all rows of $A$. Call this submatrix
    %$M$ as well. Notice that, $M$ has only homogeneous columns by the assumption that
    %$A$ consists of homogeneous columns only.
    %If column $j$ becomes non-homogeneous in $A_{\out}$, then 
    By Proposition~\ref{prop:homogeneity_loss}, there is a set of
    columns $s_1<j,\ldots, s_k<j$, $s_i\simeq j$,
    whose additions to $j$ by {\sc ConMat} make
    it non-homogeneous for the first time. The columns $s_i$, $1\leq i \leq k$,
    are also available for a complete reduction of $A$ in addition to possibly
    other column additions. This implies that we have 
    $\low_{\bar A}(j)\leq \low_{A_{\out}(j)}$ showing that $j$ becomes non-homogeneous in $A_{\out}$ as well.

    Next, consider that column $j$ which is homogeneous in $A$ becomes non-homogeneous
    in $\bar A$ by a complete reduction. We show that $j$ is made non-homogeneous
    by {\sc ConMat} in $A_{\out}$ also. By Proposition~\ref{prop:homogeneity_loss}, there is a set of columns $s_1<j,\ldots, s_k<j$, $s_i\simeq j$, whose additions to $j$ by the complete reduction of $A$ make it non-homogeneous for the first time.
    These set of columns are also available for {\sc ConMat} to reduce $j$.
    Indeed, steps ** in {\sc ConMat} adds these columns to $j$ before any other
    column from a different Morse set is added to $j$. By Observation~\ref{obs:non-homogeneous-columns}, these additions do not change the non-homogeneity of $j$
    in $A_{\out}$.
 \end{proof}
}

\begin{prop}\label{prop:hom_coincides_with_complete_reduction}
    Consider the matrix $A_{\out}$ computed by {\sc ConMat} and its complete reduction 
    $\compl{A}_{\out}$.
    A column $j$ is homogeneous in $A_{\out}$ if and only if $j$ is homogeneous in $\compl{A}_{\out}$.
\end{prop}
\begin{proof}
    Suppose on the contrary a homogeneous column $j$ in $A_{\out}$ becomes non-homogeneous during the complete column reduction of $A_{\out}$.
    Assume additionally, that $j$ is the first such column that loses homogeneity.
    Proposition \ref{prop:homogeneity_loss} implies that just before the first change of the pivot of $j$, a homogeneous column $s$ such that $s\simeq j$ 
    has been added to $j$.
    Let $A'$ denote the matrix immediately before this addition.
    If $\low_{A_{\out}}(s)=\low_{A'}(s)$ then, also by Proposition \ref{prop:homogeneity_loss}, we have 
        $\low_{A_{\out}}(s)=\low_{A'}(s)=\low_{A'}(j)=\low_{A_{\out}}(j)$,
    which contradicts property R3 of $A_{\out}$. 
    Otherwise, we can find an earlier stage of the 
    complete column reduction with corresponding matrix $A''$ when a homogeneous column $s'$ is to be added to $s$ changing its pivot.
    Again, if $\low_{A_{\out}}(s')=\low_{A''}(s')$, we get a contradiction 
    by the same argument.
    Since we have a finite number of column additions, the 
    recursive argument has to stop giving us a contradiction.

    The reverse implication follows directly by Observation \ref{obs:non-homogeneous-columns} as a non-homogeneous column cannot become homogeneous.
    \qed
\end{proof}

Now, we prove
that {\sc ConMat} satisfies R2. We use the following notations.
As {\sc ConMat} keeps
adding columns from the left, the input matrix $A$ keeps changing. Denoting the input
matrix $A$ as $A_0$, let $A_j$ be the matrix
after processing the column $j$. 
Every column $s$ in $A_j$ represents a $p$-chain $D_s^{A_j}$ if $\sigma_s$ is a $p$-simplex. Initially, $D_s^{A_0}=\langle \sigma_s \rangle$ in $A_0$. After the column $s$ is processed, its chain changes. In particular, $D_s^{A_j}=D_s^{A_0}$ if $s> j$ and
$D_s^{A_j}=(\sum_{i}D_{j_i}^{A_j})+\langle\sigma_s\rangle$ where
columns $j_1<j_2<\cdots <j_t$ had been added to column $s\leq j$. 
We say a chain has \emph{support} on a simplex $\sigma$ if the chain has
non-zero coefficient on $\sigma$.
The \emph{content} of a column $s$ in $A_j$ represents
the boundary $(p-1)$-chain $\partial D_s^{A_j}$. Specifically, $A_j[i,s]=1$ if and only if
$\partial D_s^{A_j}$ has support on $\sigma_i$.
In analogy to our notation $\low_A(s)$ for a column $s$
in a matrix $A$, we denote $\low(D)$ to be the largest index $i$ so that
chain $D$ has support on $\sigma_i$. 

\begin{prop}\label{prop:Aout_is_R2}
    Matrix $A_{\out}$ computed by {\sc ConMat} satisfies %\emph{R2}. 
    \emph{R1}, \emph{R2}, and \emph{R3}.
\end{prop}
\begin{proof}
As we have already mentioned that R1 and R3 easily follow from the algorithm, only R2 requires non-trivial argument.

We prove by induction that, as the algorithm processes columns from left to right,
it never makes an already processed homogeneous column targetable.
Assume inductively that before processing a column $j$, there was no targetable column $s<j$ that is homogeneous. 
We show that after processing column $j$, the assertion remains true.
To prove this by contradiction, assume that after processing column $j$, a homogeneous column $s<j$ has become targetable.
Since in matrix $A_j$ no column $k>j$ has been modified yet, without loss of the generality, we can assume that $j$ is the last column of the matrix, that is $A_j=A_{\out}$.

\begin{enumerate}
    \item\label{it:j_is_a_cycle} Since $s$ is a targetable column for $j$, $\partial D_j^{A_j}$ is a cycle in $K$ with $\low(\partial D_j^{A_j})=s$.
    \item\label{it:cDs_is_a_cycle} Consider $\compl{A}_j$ -- a \emph{complete reduction} of $A_j$.
        Point(\ref{it:j_is_a_cycle}.) implies that $\sigma_s$ is the last simplex completing a cycle.
        This means $\compl{D}_s:=D_s^{\compl A_j}$ is a cycle in $K$. 
        In particular, $\partial\compl{D}_s=0$, which means that column $s$ in $\compl{A}_j$ consists only of zeros.
    \item Point(\ref{it:cDs_is_a_cycle}.) implies that $s$, as an empty column, is no longer a homogeneous column in $\compl{A}_j$, but $s$ is homogeneous in $A_j$, which contradicts Proposition \ref{prop:hom_coincides_with_complete_reduction}.
    \qed
\end{enumerate}
\end{proof}

\cancel{
\begin{prop}
    Algorithm {\sc ConMat} satisfies \emph{R1}, \emph{R2}, and \emph{R3}.
    \michal{We can merge this proposition with Proposition \ref{prop:Aout_is_R2} as suggested.}
    \label{prop:R1R2R3}
\end{prop}
\begin{proof}
The fact that
$A_{out}$ satisfies R1 and R3 is immediate from the execution of {\sc ConMat} and can be rigorously
proved with the same argument as for Propositions \todo{**} in~\cite{DLMS24}. Proposition~\ref{prop:Aout_is_R2} above
shows that it satisfies R2 also.
\end{proof}
}
\begin{theorem}
    Given an $n\times n$ matrix of a filtered chain complex $(C,d)$ in a fixed $d$-admissible basis on input, {\sc ConMat}
    outputs a connection matrix of $(C,d)$ in $O(n^3)$ time. 
\end{theorem}
\begin{proof}
    Let $A'$ be the input matrix $A$ at any stage of {\sc ConMat}.
    Suppose that after addition of a column $s$ to $j$ we also made the corresponding row addition of $j$ to $s$. 
    Observe that the row addition
        \textbf{(i)} affects only entries of the row $s$, corresponding to a homogeneous column,
        \textbf{(ii)} does not change the pivot of any column $k$, because if $A'[j,k]=1$ then $s\neq \low_{A'}(k)$ because $s<j$,
        \textbf{(iii)} does not affect further additions, because by Proposition \ref{prop:Aout_is_R2}, column $s$ is not a pivot for any other column.

        Thus, we could perform all corresponding row additions, and, by \textbf{(iii)} the sequence of column operations remains the same as without them.
        In particular, let $V$ and $U$ denote the matrices encoding column and row additions, respectively.
        Then, we have $A_{\out}=A V$ and $A_{\out}^r:=U A V$.
        Since $A_{\out}$ is reduced by Proposition~\ref{prop:Aout_is_R2} (satisfies R1, R2, and R3), it is easy to conclude from \textbf{(i)}, \textbf{(ii)}, and \textbf{(iii)} that $A_{\out}^r$ is reduced as well.
        Moreover, the columns and rows of $A_{\out}^r$ are expressed with compatible basis, which allow us to apply \cite[Proposition 4.9]{DLMS24} inductively to extract columns and rows corresponding to 
        $J:=\II_n\setminus J_{h}(A)\setminus J_{t}(A)$
        to obtain a connection matrix of $(C,d)$. 

        Notice that due to \textbf{(ii)} we have $J_{h}(A_{\out})= J_{h}(A_{\out}^r)$ and $J_{t}(A_{\out})= J_{t}(A_{\out}^r)$.
        Moreover, by \textbf{(i)} and \textbf{(iii)}, matrices $A_{\out}^r$ and $A_{\out}$ restricted to rows and columns corresponding to $J$ are equal.
        Therefore, {\sc ConMat} computes the connection matrix without the row additions.

        The inner \textbf{for} loop of {\sc ConMat} runs in $O(n^2)$ time (similar to the persistence algorithm) for the column additions.
        Thus, in total the algorithm has $O(n^3)$ time complexity. \qed
\end{proof}

\cancel{
\begin{theorem}
Given an $n\times n$ matrix of a filtered chain complex $(C,d)$ in a fixed $d$-admissible basis on input, {\sc ConMat}
outputs a connection matrix of $(C,d)$ in $O(n^3)$ time. 
\label{thm:main}
\end{theorem}
\begin{proof}
In algorithm {\sc ConMat},
we made two main changes, (i) there are only left-to-right column additions, (ii) there are no row additions as in step 8 of {\sc ConnectMat}. First, we observe that even if we had added the row additions in {\sc ConMat}, it would not have changed the extracted connection matrix. This is because for every column $s$ being added to a column $j$ in {\sc ConMat}, we have $s<j$ where $s$ is a homogeneous column. 
It would have triggered addition of the row $j$ to row $s$. For any column $k$ this row addition does not
change the pivot row $\low_A(k)$ since $j>s$.
This means that no row addition would have made a difference in the homogeneity of the columns (this is the not the case for {\sc ConnectMat}). 
Next, observe that all left-to-right column and bottom-to-top row additions can be summarized 
as multiplying $A$ with an
upper triangular matrix $V$ from right and a lower triangular matrix $U$ from left respectively, that is,
$A_{out}=UAV$. This means that we could have done all row additions after all column operations without affecting $A_{out}$. In that case, since row additions only affect the rows whose corresponding columns are homogeneous, the final connection matrix which eliminates the homogeneous columns and rows is not affected by the row operations. Therefore, eliminating the row additions
in {\sc ConMat} does not affect the final output.

Given the above observations, the argument in Theorem 4.2 of \cite{DLMS24} only requires us to prove that $A_{out}$ satisfies
R1, R2, and R3 after all left-to-right column reductions which follows from Proposition~\ref{prop:R1R2R3}.
\end{proof}
}
%\michal{We should make an observation (maybe a Proposition) that the order of Morse sets (as long as it preserve the poset) does not affect the output, only the order within Morse sets matters.}

%\cancel{
\subsection{Invariance to Morse order}
In this section, we make an observation that may be useful on its own right.
We observe that the order of Morse sets (as long as it preserves the poset and keeps the order within each Morse set) does not affect the output connection matrix; see Theorem~\ref{thm:morsefixed} for a precise
statement. The fact that the order within Morse sets matters is shown with
the example in Figure~\ref{fig:annulus}. 
Indeed, in~\cite{DLMS24}, it is shown that two different orders of the simplices in the Morse set for the orbit gives two connection matrices violating Theorem~\ref{thm:morsefixed}.

\begin{definition}
%Let $L=p_0,\ldots,p_m$ and $L'=p_0',\ldots,p_m'$ be any two linear extensions of $P$.
Let $A$ and $A'$ be two possible input $P$-filtered matrices for {\sc ConMat}.
%a chain complex $(C,d)$ induced
%by a Morse decomposition.
We say $A$ and $A'$ are
\emph{Morse fixed} if the elementary chain bases $(\langle \sigma_{1}^p\rangle,\ldots, \langle \sigma_{k}^p\rangle)$ 
of $A$ and $A'$ restricted to every Morse set $M_p$, $p\in P$,
have the same linear order. We write $i=_\sigma i'$ if $\langle \sigma_i\rangle$
and $\langle \sigma_{i'}\rangle$ represent the same basis elements
in the bases for $A$ and $A'$ respectively. We also write $[i]_P$ to denote
the point $p\in P$ so that $\sigma_i\in M_p$.
\end{definition}

Notice that $A$ and $A'$ that are Morse fixed may differ in the ordering of
the columns and rows across Morse sets because they may correspond to
two different linear extensions of the poset $P$. Proposition~\ref{prop:morsefixed} essentially
says that the output of {\sc ConMat} does not change on algebraical level with different linear extensions
of $P$ as long as they do not alter the ordering of simplices within
each Morse set. We prove Proposition~\ref{prop:morsepath} that helps us to
prove Proposition~\ref{prop:morsefixed}.

\begin{prop}
    For every stage of {\sc ConMat}, $A_i[l,j]\neq 0$ implies $[l]_P\leq_P [j]_P$.
    Moreover, if {\sc ConMat} adds column $k$ to column $j$ during its execution, then
    $[k]_P\leq_P [j]_P$.
    \label{prop:morsepath}
\end{prop}
\begin{proof}
    We prove the proposition by induction.
    The claim is true for $i=0$ because $A_0$ represents the boundary homomorphism that is $P$-filtered.
    In particular, if $A_0[l,j]\neq 0$ then there is a trivial path from $\sigma_j$ to $\sigma_l$ ($\sigma_l\in F_\cV(\sigma_j)$);
        therefore $[l]_P\leq[j]_P$.
    Assume that $A_{i-1}$ satisfies the assumption.
    Note that by passing to $A_{i}$ only $i$th column is modified. 
    Let $k$ be a column that is added to column $i$.
    Let $l$ be a row with a non-zero entry in the column $i$ after the addition of column $k$.
    Then, either the entry becomes non-zero as the result of the addition ($A_{i-1}[l,i]= 0$ and $A_{i}[l,i]\neq 0$), or it was non-zero already before the addition ($A_{i-1}[l,i]=A_{i}[l,i]\neq0$).
    The later case follows from the inductive assumption.
    To see the former case, note that 
        if $m$ is the pivot of $k$, then necessarily, $A_{i-1}[m,k]=A_{i-1}[m,i]\neq 0$.
    Therefore, by the inductive assumption $[m]_P\leq [i]_P$.
    Since $k$ is a homogeneous column, we have $[k]_P=[m]_P$.
    If the addition of $k$ made the entry $A_{i}[l,i]$ non-zero, 
        we have $A_{i-1}[l,k]\neq 0$, implying $[l]_P\leq_P [k]_P$.
    In consequence, $[l]_P\leq_P [k]_P=[m]_P\leq_P[i]_P$.
    \qed
\end{proof}

%\tamal{The following cannot be true...something is wrong in the proof. I shall just state what i think is correct and not try to prove it now because this part is not going into the submission.}
\begin{prop}
Let $A_{out}$ and $A'_{out}$ be the matrices produced by {\sc ConMat} after
reducing Morse fixed input matrices $A$ and $A'$ respectively. Then,
  $A_{out}[i,j]=A'_{out}[i',j']$ where $i=_\sigma i'$ and $j=_\sigma j'$. 
  Furthermore, if $\ell$ is the pivot row of a homogeneous column $j$ in $A_{out}$,
  then $\ell'=_\sigma \ell$ is the pivot row of the homogeneous column $j'$ in $A_{out}'$.
\label{prop:morsefixed}
\end{prop}
\begin{proof}
    We prove the claim by inducting on $j$ where $A_j$ represents
    the processed matrices $A$ up to
    column $j$.
    
    The first column ($j=0$) corresponds to a vertex,
    thus has no non-zero entry in both matrices $A_{out}$ and $A'_{out}$.
    % \michal{The above sentence would be true for any Lefschetz complex, in particular, the first column would not have to correspond to a vertex. 
    %     This made me think, is there any place in the paper where we make use of the assumption (made at the beginning of Section \ref{sec:Cvec}) that $K$ is a simplicial complex?
    %     Maybe there is no point of changing the assumption, but I'm asking out of curiosity.}
    This confirms the base case. For induction, assume that the
    hypothesis is true for $A_{j-1}$. Consider processing the column $j$
    in $A$. Let $j'=_\sigma j$ be the column in $A'$ corresponding to column
    $j$ in $A$. If the algorithm {\sc ConMat} adds a column $k$ to
    column $j$, we have (i) $k <j$,  (ii) 
    % $[k]_P \leq_P [j]_P$
    $[k]_P \leq [j]_P$
    by Proposition~\ref{prop:morsepath}, and (iii) column $k$ is homogeneous. Let $k=_\sigma k'$. Then, 
    $[k']_P=[k]_P\leq_P [j]_P=[j']_P$ 
    % $[k']_P=[k]_P\leq_P [j]_P=[j']_P$ 
    and hence $k'< j'$ in $A'$.
    By inductive hypothesis,
    $A_{out}[i,k]= A'_{out}[i',k']$ for all $i=_\sigma i'$.
    Also, since $k$ is a homogeneous column,
    so is $k'$ by inductive hypothesis and their pivot rows, say $\ell$ and $\ell'$
    satisfy $\ell=_\sigma \ell'$. 
    It follows that when {\sc ConMat} processes the column $j'$, it also
    adds column $k'$ to $j'$. Therefore, {\sc ConMat} adds every column 
    to column $j'$ while processing $A'$ whose corresponding columns it adds to $j$ while processing
    $A$. Reversing the roles of $A$ and $A'$, we can claim that {\sc ConMat}
    adds every column to $j$ while processing $A$ whose corresponding columns it adds to $j'$ while
    processing $A'$. Essentially, this means that $A_{out}[i,j]=A'_{out}[i',j']$
    where $i=_\sigma i'$ and $j=_\sigma j'$. Furthermore, since the order
    of the simplices within each Morse set remains unaltered, the pivot rows of $j$ and $j'$ (if they remain homogeneous) 
    correspond to the
    same simplex completing the induction. \qed
\end{proof}

\begin{theorem}
    Let $S$ and $S'$ be two connection matrices output by {\sc ConMat} from
    the two Morse fixed input matrices $A$ and $A'$ respectively. Then, $S[i,j]=S'[i',j']$ where $i=_\sigma i'$ and $j=_\sigma j'$.
    \label{thm:morsefixed}
\end{theorem}
\begin{proof}
    It follows from Proposition~\ref{prop:morsefixed} that a column (row) $j$ remains homogeneous or targetable in $A_{out}$ if and only if
    $j'$ with $j'=_\sigma j$ is homogeneous or targetable in $A_{out}'$. Then, the
    claim follows from the fact that $A[i,j]=A'[i',j']$ where $i=_\sigma i'$ and
    $j=_\sigma j'$. \qed
\end{proof}
The following corollary is immediate for the two connetion matrices $S$ and $S'$ mentioned in the above theorem.
\begin{corollary}
    There exists a permutation matrix $R$ so that $R^T S R=S'$.
\end{corollary}
If the input matrices $A$ and $A'$ are not Morse fixed, the above corollary may or may not hold. For the example in Figure~\ref{fig:matrixalgo}, one can verify that the output connection matrices are the same up to a permutation even though input matrices are not kept Morse fixed; see the two connection matrices shown in~\cite{DLMS24}.
However, the example shown below establishes that the two connection matrices may not
be equivalent up to a permutation.
%\tamal{Correct proposition}
%\begin{prop}
%Let $A_{out}$ and $A'_{out}$ be the matrices produced by {\sc ConMat} after
%reducing Morse fixed input matrices $A$ and $A'$ respectively. Then,
%  a column $j$ is homogeneous in $A_{\out}$ if and only if
 % column $j'$ is homogeneous in $A'_{out}$ where $j=_\sigma j'$.
% \label{prop:morsefixed}
%\end{prop}
\begin{figure}[htbp]
\centerline{\includegraphics[width=0.75\textwidth, keepaspectratio]{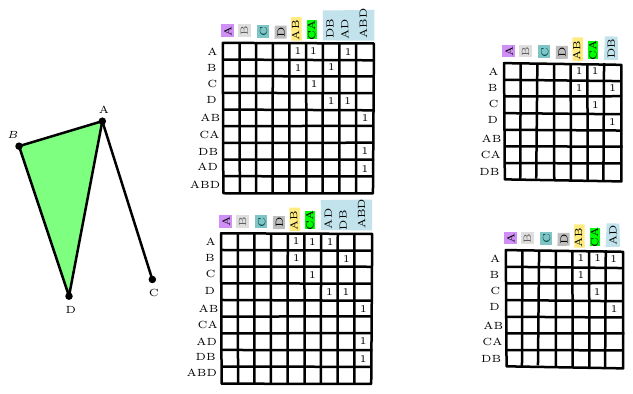}}
\caption{(Minimum) Morse decomposition $\{\{A\},\{B\},\{C\},\{D\}, \{AB\}, \{CA\},\{AD,DB,ABD\}\}$, (middle) two matrices with two different orders on the Morse set
$\{AD,DB,ABD\}$, (right) two connection matrices which cannot be transformed to the other by a permutation.
}
\label{fig:PermutMorseFix}
\end{figure}
\cancel{
\tamal{\section{Updating connection matrix}
There are three operations: (i) Extend a Morse set $M$ on right by one
more column, (ii) Extend $M$ on left by one column, (iii) Exchange two columns within the same Morse sets.
\begin{remark} Exchanging two columns across two Morse sets can be done using the three operations. Let the column $c_1$ on left belong to Morse set $M_1$
which is to be exchanged with the column $c_2$ on right belonging to Morse
set $M_2$. First, extend $M_1$ on right, exchange $c_1$ and $c_2$, extend $M_2$ on left. 
\end{remark}
(i) Extend $M$ on right: Let $c$ be the column just after $M$ on right.\\
{\bf Case a}: $c$ is non-homogeneous: After including $c$ into $M$, if it remains to
be non-homogeneous, nothing needs to be done. Otherwise, $c$ has now become
homogeneous.
\begin{itemize}
\item  Take out the column $c$ and its target column $\low(c)$
from the connection matrix. 
\item $c$ could be added to columns on right that has a $1$ in the
row $\low(c)$. We collect all such columns and add $c$ to them. We claim
that, for a column $c'$ no column other than $c$ needs to eb added. If this were not true, after addition with $c$, the column $c'$ would have a $1$ in a row that is equal to $\low(b)$ where $b$ is a homogeneous column to the left of $c$. The column $c'$ aquired this $1$ after addition with $c$, so the column $c$ had a $1$ in the row $\low(b)$ which should have triggered an addition of $b$ to $c$ to eliminate that $1$.
\item Before moving $c$, it could be a target of a homogeneous column $d$, that is $\low(d)=c$. We make $d$ non-homogeneous and update the connection matrix. We add $d$ to any column $d'$ on right to which $d$ was added before. The columns like $d'$ can be determined from the $V$ matrix. We consider the status of $d'$ before and after the addition with $d$. 
\begin{itemize}
\item First assume $d'$ was homogeneous.
Then, it cannot be the case that $\low(d)=c>\low(d')$ before because then column for 
$\low(d')$ will be to the left of $c$ and then $d'$ could not be homogeneous. Also, $\low(d)=\low(d')$ is not allowed because in that case addition of $d$ to $d'$ before would have caused $\low(d')<\low(d)$. So, the only possibility is that $\low(d)<\low(d')$. Then, adding $d$ back to $d'$ cannot make $d'$ non-homogeneous and also cannot change $\low(d')$.
Thus, $d'$ does not have to added to any more columns to its right because of being changed by addition of $d$ to it.
\item Now assume that $d'$ was not homogeneous. Since $d$ was homogeneous, $\low(d)>\low(d')$. Then,
by adding $d$ to $d'$, we keep $d'$ non-homogeneous because $\low(d')=\low(d)=c$ after the addition.
\end{itemize}
\end{itemize}
{\bf Case b}: $c$ is homogeneous: This is not possible. The leftmost column of any
Morse set cannot be homgeneous because $\low(c)$ belongs to a Morse set to the left of $M$ and cannot belong to $M$.}\\

\tamal{(ii): Extend $M$ on left: Let $c$ be the column just before $M$ on left.\\
{\bf Case a}: $c$ is non-homogeneous. In this case $c$ remains to be non-homogeneous
after being included into the Morse set $M$. So, nothing needs to be done with column $c$. However, including $c$ into $M$ may make some other column $c'$ in $M$ to be homogeneous. Then, $c'$ needs to be added to all columns conflicting with it. Using the previous arguments, one can show that no more additions will be necessary. In this case, take out $c'$ and its target column out of the connection matrix.\\
{\bf Case b}: $c$ is homogeneous. In this case $c$ necessarily becomes
non-homogeneous after being included into $M$.
\begin{itemize}
    \item  We make $c$ non-homogeneous and its target non-targettable; update the connection matrix by bringing these two columns. 
    \item We find out the columns $c'$ to which $c$ were added. This can be found out by maintaining the $V$ matrix. Add $c$ to these columns. We claim that no more columns need to be added to $c'$ because $c$ is made non-homogeneous. If there is a homogeneous column $b$ with $c'$ having a $1$ in the row $\low(b)$ after addition with $c$, then $c$ also has a $1$ in the row $\low(b)$. This means $b$ belongs to a Morse set to the left of $M$ because $c$ belonged to a Morse set to the left of $M$ before moving to $M$. We reach a contradiction because $b$ being a homogeneous column would have been added to $c$ to eliminate the $1$ in row $\low(b)$. 
\begin{itemize}
\item observe that if $c'$ were homogeneous, it remains to be homogeneous 
with the same $\low(c')$ because $\low(c) < \low(c')$ (because $\low(c')$ is in $M$ and $\low(c)$ is to the left of $M$.
\item If $c'$ were non-homogeneous, adding $c$ to it cannot make it homogeneous because the new $\low(c')$ can at best be moved down to $\low(c)$ which keeps $c'$ non-homogeneous.
\item Above observations mean that no changes are required for the connection matrix due to additions of $c$ to $c'$.
\end{itemize}
    \item It may happen that by bringing $c$ to $M$, some other column $c'$ becomes homogeneous. Identify $c'$ and add $c'$ to all columns conflicting with it. Using the same argument as above, one can show that no more additions are necessary. Again, take out $c'$ and its target from the connection matrix.
\end{itemize}
(iii): Exchange column $c$ with the column $c'$ immediate to its right but within the same Morse set.\\
{\bf Case a}: If $c$ is non-homogeneous, nothing needs to be done.\\
{\bf Case b}: If $c$ is homogeneous and $c$ was not added to $c'$, nothing needs to be done.\\
{\bf Case c}: If $c$ is homogeneous and $c$ was added to $c'$.
\begin{itemize}
    \item Add $c$ to $c'$. No change in connection matrix. We claim that no more
column additions are required. Suppose that there is a homogeneous column
$b$ that needs to be added to $c'$ now after $c$ is added to $c'$. In that case, $c$ would have a $1$ in the row $\low(b)$. But, then $b$ should have been added to $c$ to eliminate its $1$ in the row $\low(b)$, a contradiction.
\end{itemize}
In all cases, we have at most $O(n)$ column additions each requiring at most $O(n)$ time giving us an $O(n^2)$ time complexity.}
}

\section{Persistence of Conley complex and Morse decomposition} \label{sec:persistence}
In this section, given a suitable function $f:K\rightarrow \mathbb{R}$, first we define a concept of persistence for Conley complexes induced by a Morse decomposition $\mathcal M_{\mathcal V}$ of a vector field $\mathcal V$ on $K$. Then, we derive that
this persistence is the same for all Conley complexes for $\mathcal M_{\mathcal V}$. Thus, the
persistence of any Conley complex induced by $\mathcal M_{\mathcal V}$ defines the persistence of $\mathcal M_{\mathcal V}$ under $f$.

\cancel{
    In order to have a persistence of the Conley complex we need to assume that combinatorial (multi)vector field comes with a function $f$ defined on the simplicial complex $K$. (i) For Forman gradient vector field, we
    assume the discrete Morse function with simplices in the doubleton having the same value. (ii) For multivector field in general, we assume that the function values for all simplices in a Morse set are the same.
 }   
Let $\md_{\mathcal V}=\{M_p\,|\, p\in P\}$ be a Morse decomposition indexed by a poset $\mathbb{P}=(P,\leq_P)$ for a (multi)vector field $\mv$ on a simplicial complex $K$. 
Let $(C_{\md_{\mathcal V}}, d_{\md_{\mathcal V}})$ be the $\mathbb{P}$-filtered chain complex
given by the Morse decomposition $\md_{\mathcal V}$ and $(\bar{C}_{\md_{\mathcal V}},\bar{d}_{\md_{\mathcal V}})$ be one of its
Conley complexes as defined in Section~\ref{sec:concomplex}. 
%In what follows, we find it easier to work with the
%opposite category $\mathbb{P}^{op}=(P^{op},\leq_{P^{op}})$. We have $p\leq_P q$ if and only if
%$q\leq_{P^{op}} p$.

\begin{definition}
    We say a function $f: K\rightarrow \mathbb{R}$ is $\md_\mv$-Lyapunov if (i) for every $p\in P$,
    and every pair $\sigma_1,\sigma_2\in M_p$, $f(\sigma_1)=f(\sigma_2)$, (ii) for any pair 
    $p \leq_{P} q$, $\sigma\in M_p$ and $\sigma'\in M_q$ imply $f(\sigma)\leq f(\sigma')$.
\end{definition}
\begin{definition}
Any function $f:K\rightarrow \mathbb{R}$ induces a \emph{chain function} $f_C: C\rightarrow \mathbb{R}$
on any chain complex $(C,d)$ defined with simplices in $K$ by the assignment that,
%where $C$ is a subspace of the chain space of $K$ and $d$ is the
%restriction of the boundary homomorphism of $K$ to this subspace. 
for a chain $c\in C$, we have
$f_C(c)=\max_{\sigma\in {\mathrm supp}(c)} f(\sigma)$. 
\label{def:chainf}
\end{definition}
%The Conley complex $(\bar{C},\bar{d}):=(\bar{C}_\md,\bar{d}_\md)$
%computed from the input chain complex $(C_\md,d_\md)$ satisfies that $\bar{C}$ is
%a subspace of the chain space of $K$. Furthermore, $(\bar{C},\bar{d})$ is $\mathbb{P}$-filtered
%implying that $\bar{C}=\bigoplus_{p\in P} \bar{C}_p$. 
Let $f$ be a
$\md_\mv$-Lyapunov function and $(C,d):=(C_{\md_\mv},d_{\md_\mv})$. It follows
that $f_C(c)$ is the same for every $c\in C_p$, $\forall p\in P$.
%be the $P$-filtered chain complex
%induced by the Morse decomposition $\md$ given by $\mathcal V$. 
The following is obvious from the definitions.
\begin{prop}
$p\leq_{P} q \implies f_C(C_p) \leq f_C(C_q)$.
\label{prop:inducedf}
\end{prop}
\begin{proof}
    The chain space $C_p$ for any $p$ must contain the elementary chain
    $\langle \sigma\rangle$ for a simplex $\sigma \in M_p$. Then, 
    $f_C(C_p)=f(\sigma\in M_p)\leq f(\tau\in M_q)=f_C(C_q)$.
    \qed
\end{proof}
Proposition~\ref{prop:inducedf} allows us to introduce the following definition.
\begin{definition}
Let $m=|P|$. Then, a total order
$p_1\leq p_2\leq \cdots\leq p_m$ that
extends the poset $P$
satisfying  
$f_C(C_{p_i})\leq f_C(C_{p_j})$ for $i\leq j$ is called an $f$-compatible
order $P_f$.
\end{definition}

Consider an $f$-compatible order $P_f: p_0\leq \cdots\leq p_m$ and the chain complex $(C,d):=(C_{\md_\mv},d_{\md_\mv})$. For each $p_i\in P$, consider the chain space $\CC_{p_i} := \oplus_{j\leq i}{C}_{p_j}$ with $f(\CC_{p_i}):=f_C(C_{p_i})$
and the chain complex $(\CC_{p_i},\oplus_{j\leq i} d_{p_ip_j})$ which is the restriction of $(C,d)$ on the
downset of $p_i$ in $P_f$. Clearly, $\CC_{p_i}$ is a subspace of $\CC_{p_j}$ where $p_i\leq_{P_f} p_j$, giving
us a filtration
\begin{equation*}
(\FF,P_f): \CC_{p_0}\subset \CC_{p_1}\subset\cdots\subset\CC_{p_m}
\end{equation*}
Similarly, define the filtration derived from a Conley complex $(\bar{C}, \bar{d})$
\begin{equation*}
(\bar\FF,P_f): \bar\CC_{p_0}\subset \bar\CC_{p_1}\subset\cdots\subset\bar\CC_{p_m}
\end{equation*}
%\michal{It follows from the algorithm, but it would be nice to comment why we it is true that we have these inclusions.}\tamal{I have explained why the subspace relation holds.}
where 
$\bar\CC_{p_i}:=\oplus_{j\leq i}\bar{C}_{p_j}$ and 
$(\bar\CC_{p_i},\oplus_{j\leq i}\bar{d}_{p_ip_j})$ is
a chain complex for each $p_i\in P$.

We obtain persistence modules
by taking the homology functor:
\begin{equation*}
H({\FF},P_f): H_*(\CC_{p_0})\rightarrow H_*(\CC_{p_1})\rightarrow\cdots\rightarrow H_*(\CC_{p_m})
\end{equation*}

\begin{equation*}
H(\bar{\FF},P_f): H_*(\bar\CC_{p_0})\rightarrow H_*(\bar\CC_{p_1})\rightarrow\cdots\rightarrow H_*(\bar\CC_{p_m})
\end{equation*}
The persistence $\pers(\FF,P_f)$ and $\pers(\bar\FF, P_f)$ given by the above persistence modules are the persistence of the chain complexes $(C,d)$ and $(\bar{C},\bar{d})$ respectively w.r.t. the linearization $P_f$. 
% Observe that the connection matrix $\bar{A}$ of the Conley complex is the
%filtered boundary matrix of the filtration $\bar\FF$. Therefore, applying the 
%persistence algorithm on $\bar A$ provides the persistence pair.
\cancel{
We observe that the above persistence can be computed from the
original input filtered boundary matrix $A$ from which we obtained the connection matrix.
It turns out that a complete reduction of $A$ indeed provides the persistence
of the Conley complex. First, similar to the chain complex
$(\bar\CC_{p_i},\oplus_{j\leq i} \bar{d}_{p_ip_j})$, define the chain complex
$(\CC_{p_i},\oplus_{j\leq i} {d}_{p_ip_j})$
from the chain complex $(C,d)$.
}
We observe that two persistence modules $H(\FF,P_f)$  and $H(\bar\FF,P_f)$
have pointwise isomorphic vector spaces because of homotopy equivalences.
\begin{prop}
    $\CC_{p}\cong \bar{\CC}_p$ and hence $H_*(\CC_p)\cong H_*(\bar{\CC}_p)$ for every $p\in P_f$.
    \label{prop:pwise}
\end{prop}
Proposition~\ref{prop:pwise} leads to the isomorphism of the two modules.
\begin{prop}
   $H(\FF,P_f)\cong H(\bar\FF, P_f)$. 
\end{prop}
\begin{proof}
    We observe that the following diagram commutes. This is because
    the vertical morphisms
    are isomorphisms due to Proposition~\ref{prop:pwise} and the
    horizontal morphisms are induced by inclusions in chain complexes.
\begin{equation*}
%$$
\xymatrix
{
H\FF:\ar@{=}[d]^{\simeq} & \homo_*(\CC_{p_0}) \ar@{->}[r] \ar@{=}[d]^{\simeq}
%& \  \homo_p(S_0) \  \ar@{->}[d]^{\mysimeq}
& \ \homo_*(\CC_{p_1}) \  \ar@{->}[r]\ar@{=}[d]^{\simeq}
%& \ \homo_p(S_1) \   \ar@{->}[d]^{\mysimeq}
& \ \homo_*(\CC_{p_2}) \  \ar@{->}[r]\ar@{=}[d]^{\simeq}
& \ \ldots \ldots \ar[r] \
& \  \homo_*(\CC_{p_n}) \ar@{=}[d]^{\simeq}
\\
H\bar\FF: & \homo_*(\bar\CC_{p_0}) \ar@{->}[r]
%& \ \homo_p(T_0) \  
& \ \homo_*(\bar\CC_{p_1})\  \ar[r]
%& \ \homo_p(T_1)\   \ar@{<-}[r]
& \ \homo_*(\bar\CC_{p_2})\ \ar@{->}[r]
& \ \ldots \ldots\ \ar[r]  
& \ \homo_*(\bar\CC_{p_n})
}
%$$
\label{eq:isomorph}
\end{equation*}
\end{proof}
%\end{center}
%\caption{}
%\label{eq:isomorphmod}
%\end{figure}

Let $K_i:=\sqcup_{j\leq i} M_{p_j}$ be the subcomplex of the input complex $K$
containing simplices in every Morse set $M_{p_j}$ for $j\leq i$. Letting
$C(K_i)$ denote the chain group of $K_i$, we get the following proposition.
\begin{prop}
    Chain complexes $(C(K_i),\oplus_{j\leq i} d_{p_ip_j})$ and $(\CC_{p_i},\oplus_{j\leq i} {d}_{p_ip_j})$
    are identical.
    \label{prop:chaincomplex}
\end{prop}
We observe that $K_i\subseteq K_j$ for $i\leq j$ giving us a filtration
of chain spaces and the corresponding persistence module
\begin{eqnarray*}
(\mathbf{F},P_f)&:& C(K_0)\subset C(K_1)\subset\cdots\subset C(K_m)\\
H(\mathbf{F},P_f)&:& \homo_*(C(K_0))\rightarrow \homo_*(C(K_1))\rightarrow\cdots\rightarrow \homo_*(C(K_m))
\end{eqnarray*}

Proposition~\ref{prop:chaincomplex} leads to the following isomorphic persistence
modules.
\begin{prop}
    $H(\CC,P_f)\cong H(\mathbf{F}, P_f)$.
\end{prop}
\begin{proof}
The claim follows from the fact that the following diagram commutes. 
This is because
    the vertical morphisms
    are isomorphisms due to Proposition~\ref{prop:chaincomplex} and the
    horizontal morphisms are induced by inclusions in chain spaces.
\begin{equation*}
%$$
\xymatrix
{
H\CC: \ar@{=}[d]^{\simeq} & \homo_*(\CC_{p_0}) \ar@{->}[r] \ar@{=}[d]^{\simeq}
%& \  \homo_p(S_0) \  \ar@{->}[d]^{\mysimeq}
& \ \homo_*(\CC_{p_1}) \  \ar@{->}[r]\ar@{=}[d]^{\simeq}
%& \ \homo_p(S_1) \   \ar@{->}[d]^{\mysimeq}
& \ \homo_*(\CC_{p_2}) \  \ar@{->}[r]\ar@{=}[d]^{\simeq}
& \ \ldots \ldots \ar[r] \
& \  \homo_*(\CC_{p_m}) \ar@{=}[d]^{\simeq}
\\
H\mathbf{F}:& \homo_*(C(K_0)) \ar@{->}[r]
%& \ \homo_p(T_0) \  
& \ \homo_*(C(K_1))\  \ar[r]
%& \ \homo_p(T_1)\   \ar@{<-}[r]
& \ \homo_*(C(K_2))\ \ar@{->}[r]
& \ \ldots \ldots\ \ar[r]  
& \ \homo_*(C(K_m))
}
%$$
\label{eq:isomorph2}
\end{equation*}
\end{proof}
So, we have
\begin{prop}
    $H(\bar\CC,P_f)\cong H(\CC,P_f) \cong H(\mathbf{F},P_f)$.
    \label{prop:equivalence}
\end{prop}

Now consider two $f$-compatible linearizations $P_f$ and $P'_f$ of the poset $P$ and the two corresponding filtrations $(\mathbf{F},P_f)$ and $(\mathbf{F}',P'_f)$. If $f$ takes different
values on Morse sets, we have a unique linear order $P_f=P_{f'}$. Non-uniqueness is only
possible when two or more Morse sets have the same $f$-value. However, in such a case,
reordering the simplices with the same $f$-value to obtain different filtration
does not affect the persistence diagrams, see e.g.~\cite{DW22,EH2010}. Hence we have:
%We observe that the persistence diagrams of the two persistence modules remain the same.
\begin{prop}
    $\pers(H(\mathbf{F},P_f))=\pers(H(\mathbf{F}',P'_f))$.
    \label{prop:filtrations}
\end{prop}
Proposition~\ref{prop:equivalence} and Proposition~\ref{prop:filtrations} 
imply the following result.
\begin{theorem}
    $\pers(H(\bar\CC,P_f))=\pers(H(\mathbf{F},P_f))=\pers(H(\mathbf{F}',P'_f))=\pers(H(\bar\CC',P'_f)$
    \label{thm:equiv}
\end{theorem}
Theorem~\ref{thm:equiv} implies three things: \textbf{(i)} the persistence of every Conley complex for a Morse decomposition $\md_\mv$ under a given $\md_\mv$-Lyapunov function $f$ remains the same and thus
can be taken as the persistence of $\md_\mv$ under $f$, \textbf{(ii)} this
persistence can be computed by computing the persistence given by a
filtration of $K$ that is induced by $f$, \textbf{(iii)} this persistence is stable under perturbation of the
function $f$ in the infinity norm because
the bottleneck distance $d_B(\pers(H(\mathbf{F},P_f),\pers(H(\mathbf{G},P_g))\leq \|f-g\|_\infty$
by a well known result in TDA~\cite{CEH2007}.

\textbf{Computation:} Recall that the input for computing
a Conley complex is a $P$-filtered matrix $A$. We also make it
$P_f$-filtered for a linearization $P_f:p_0\leq p_1\leq \cdots\leq p_m$ of $P$
in the following sense.
Let columns $k$ and $k'$ represent simplices in $M_{p_i}$ and $M_{p_j}$ respectively.
Then $i\leq j \implies k \leq k'$. The rows are also ordered similarly.
In other words, all columns and rows for simplices in $M_{p_i}$ come before the columns and rows respectively for simplices in $M_{p_j}$ if $i\leq j$. Columns and rows for simplices within each
Morse set are ordered arbitrarily but respecting the non-decreasing order of dimensions.

    Next, compute the ordinary persistence of this filtered matrix $A$ by a complete reduction. Retain only the non-homogeneous columns and zeroed out columns that are not targeted (just like the output rule of our connection matrix computation). These reduced columns constitute the persistence pairing of $f$; see Appendix~\ref{sec:continuous} for an implementation.

\cancel{
\begin{conjecture}
    Let $f_1$ and $f_2$ be two functions compatible with the input vector field. Then, the connection matrices remain the same under fixed orders of simplices within each Morse sets: follows from the fact that arranging incomparable Morse sets in any order does not change the connection matrix. We could not make similar statements for Conley complex persistence.
\end{conjecture}

\subsection{Stability}
\tamal{This needs to be re-written. It is now obvious.}
Observe that the persistence diagram of the Conley complex is given by the pairing in the non-homogeneous columns of the reduced original boundary matrix. The homogeneous columns and targetted columns are left out. However, those columns gives pairing within Morse sets and thus resulting in zero persistence. This implies that the persistence diagram of the original filetered matrix coincides with the persistence diagram of the Conley complex.
}
%\section*{Acknowledgments.} This work is partially supported by NSF grants CCF-2049010, DMS-2301360.
\section{Experimental results and implementation}

\subsection{Empirical comparison of {\sc ConnectMat} and {\sc ConMat}} \label{sec:benchmark}

We implement algorithms {\sc ConnectMat} and {\sc ConMat} described in~\cite{DLMS24} and this work respectively. Advances made in the development of {\sc ConMat} eliminate one major obstacle in the use of existing TDA software to compute connection matrices: the interleaving of row and column operations on an input boundary matrix. However, despite this improvement {\sc ConMat} retains a dependency on column homogeneity, prompting us to implement both algorithms from the ground up.

To test performance, we generate seven multivector fields $\mv_1, \ldots, \mv_7$ over simplicial complexes $K_1, \ldots, K_7$ respectively. With the exception of $\mv_7$ and $K_7$, all field/complex pairs are generated via the following procedure: first, create a simplicial complex of the desired size; second, initialize the \textit{least coarse} multivector field over the complex, where each simplex is alone in a singleton vector; third, \textit{coarsen} the field by selecting vectors uniformly at random to merge (resolving convexity issues if necessary) until a desired \textit{connection probability} is reached - that is, a probability that given two simplicies, they belong to the same vector.
    \begin{itemize}
        \item $K_1$ consists of a single twelve-dimensional simplex, all of its faces, their faces, and so on. $\mv_1$ is a multivector field over this complex with connection probability $p \approx 0.459700$. This probability is notably higher than that of the other multivector fields. To understand why, suppose that in the creation of $\mv_1$ multivector field $\mv_1^i$ is coarsened to $\mv_1^{i+1}$ via taking the union of vectors $v$ and $v'$. Due to the density of $K_1$, there is a relatively high probability that simplex $\sigma \in v$ has a high dimensional co-face $\sigma' \in v'$, and thus that there are a large number of simplices $\tau$ where $\sigma \leq \tau \leq \sigma'$; each $[\tau]_{\mv_1^{i}}$ must also be merged with $v$ and $v'$, and this procedure repeated recursively to remove convexity violations.

        \item $K_2$ consists of a high density assortment of triangles. $\mv_2$ is a multivector field over this complex with connection probability $ \approx 0.063600$.
        \item $K_3$ is equal to $K_2$. $\mv_3$ is a multivector field over this complex with connection probability $ \approx 0.096700$
        \item $K_4$ consists of several hundred vertices, and approximately one-hundred-thousand edges; in other words, it is an extremely high density graph. $\mv_4$ is a multivector field over this complex with connection probability $ \approx 0.030900$
        
        \item $K_5$ consist of a large assortment of three, four, and five dimensional simplices and their faces. $\mv_5$ is a multivector field over this complex with connection probability $ \approx 0.000010$.
        \item $K_6$ is equal to $K_5$. $\mv_6$ is a multivector field over this complex with connection probability $ \approx 0.000030$.

        \item $K_7$ is the geometric simplicial complex described in the Appendix~\ref{sec:example_large}. Similarly, $\mv_7$ is created via the discretization process described in Appendix~\ref{sec:continuous}, and has connection probability $  \approx 0.038229$.
    \end{itemize}

For each pair $\mv_i$/$K_i$ we compute the minimum Morse decomposition $\md_{\mv_i} = \sqcup M_p$, and from this a graded boundary operator $\partial: \oplus_pC_*(M_p)\rightarrow \oplus_qC_*(M_q)$ is used to construct a filtered boundary matrix $A$. We input each $A$ to both algorithms, and time the computations, with results summarized in Table~\ref{tab:benchmarks}. Column two gives the number of simplices in each complex, while column three gives the maximum and average simplex dimensions. Column four gives the connection probability described above. Column five gives the runtime of each algorithm. Column six gives the approximate speed-up achieved in moving from {\sc ConnectMat} to {\sc ConMat}; i.e., the nearest integer to the ratio given in the prior column.

\begin{table}[H]
\centering
%\adjustbox{width=\textwidth}{
\begin{tabular}{
>{\raggedright\arraybackslash}p{.5cm} 
>{\raggedright\arraybackslash}p{1.25cm} 
>{\raggedright\arraybackslash}p{2.5cm} 
>{\raggedright\arraybackslash}p{1.5cm} 
>{\raggedright\arraybackslash}p{4cm} 
>{\raggedleft\arraybackslash}p{1.5cm}
}
\hline
$\mv_i$ & $|K_i|$ & Max / Avg Dim & Probability & {\sc ConnectMat} / {\sc ConMat} & Speed Up \\
\hline
$\mv_1$ & \texttt{008,192} & \texttt{12 / 5.500} & \texttt{0.459700} & \texttt{00h00m09s / 00h00m01s} & $\approx $\texttt{  009} \\
$\mv_2$ & \texttt{026,541} & \texttt{02 / 1.940} & \texttt{0.063600} & \texttt{00h23m48s / 00h00m02s} & $\approx $\texttt{  714} \\
$\mv_3$ & \texttt{026,541} & \texttt{02 / 1.940} & \texttt{0.096700} & \texttt{00h22m44s / 00h00m02s} & $\approx $\texttt{  682} \\
$\mv_4$ & \texttt{101,056} & \texttt{01 / 0.990} & \texttt{0.030900} & \texttt{01h09m13s / 00h01m10s} & $\approx $\texttt{  059} \\
$\mv_5$ & \texttt{055,843} & \texttt{05 / 2.744} & \texttt{0.000010} & \texttt{06h24m10s / 00h02m28s} & $\approx $\texttt{  156} \\
$\mv_6$ & \texttt{055,843} & \texttt{05 / 2.744} & \texttt{0.000030} & \texttt{20h18m20s / 00h01m51s} & $\approx $\texttt{  653} \\
$\mv_7$ & \texttt{002,481} & \texttt{02 / 1.146} & \texttt{0.038229} & \texttt{531.209ms / 049.710ms} & $\approx $\texttt{  011} \\
\hline
\end{tabular}
%}
\caption{Runtime and Speed Up for Fields $\mv_1, \ldots, \mv_7$. All tests were run on a laptop with an Intel Xeon E-2276 CPU @ 2.80GHz, 32GB RAM, and Linux operating system.} \label{tab:benchmarks}
\end{table}

We observe a significant speedup, often by several orders of magnitude, when applying {\sc ConMat} rather than {\sc ConnectMat}. This can be attributed to several factors. A greater number of operations are necessary to execute {\sc ConnectMat}, where column additions occur on both left and right, and each column addition requires a corresponding row addition. Further, these row additions have several second-order effects which preclude or complicate the use of optimizations available in {\sc ConMat}. These effects are discussed in the following paragraph.

Consider columns $i,j$ where $i<j$. {\sc ConnectMat} may, during its execution, add column $j$ to column $i$, thus triggering the addition of row $i$ to row $j$. We refer to such an operation as a \textit{downward addition}, as a low index row is added to a higher index row, below it in the matrix. Downward additions cause two primary problems. First, it is expedient to track the pivot row of each column, and a downward addition can potentially modify the pivot row of every column, causing each to be recomputed. Second, a downward addition can break the upper triangularity of the input matrix, which in {\sc ConMat} is always maintained and allows one to store and operate only on the upper right of the input matrix.

Another difficulty presented by row additions involves the structure of the input matrix. In {\sc ConMat}, where only matrix columns must be iterated over, it is an obvious choice to store columns contiguously in memory. The row additions in {\sc ConnectMat} require iterating over rows, and thus prompt the question of whether to store matrix rows or columns contiguously.

\subsection{Discretization and persistence of Morse decompositions} \label{sec:continuous}

Let $\md_\mv$ be a Morse decomposition of a multivector field $\mv$ over a simplicial complex $K$. Section~\ref{sec:persistence} describes a notion of persistence of $\md_\mv$ with respect to an $\md_\mv$-Lyapunov function $f:K\to \mathbb{R}$. In this section, we experiment with this notion of persistence through the discretization of a continuous vector field.

\begin{observation}\label{obs:downset-lyapunov}
    Let $\mv$ be a multivector field on simplicial complex $K$, and $\md_\mv = \{M_p\,|\, p\in \Poset\}$ a Morse decomposition of this field. Let $[\sigma]_{\md_\mv} = M_p$, where $M_p$ is the unique Morse set in $\md_\mv$ which contains $\sigma$. For any Morse set $M_p \in \md_\mv$, let $D_{M_p} = \cup_{q\leq p}M_q$. The \emph{downset function} $f^{\md_\mv}: K \to \mathbb{R}$ is given by $f(\sigma) = |D_{[\sigma]_{\md_\mv}}|$. This function is $\md_\mv$-Lyapunov.
\end{observation}

Observation~\ref{obs:downset-lyapunov} is easily justified. Let $\md_\mv=\{M_p\,|\, p\in P\}$ be any Morse decomposition of multivector field $\mv$. For all $p \in P$ and $\sigma_1, \sigma_2 \in M_p$, we observe $D_{[\sigma_1]_{\md_\mv}} = D_{[\sigma_2]_{\md_\mv}}$ and thus $f^{\md_\mv}(\sigma_1) = f^{\md_\mv}(\sigma_2)$. Additionally, for pair $p\leq_P q$, and simplices $\sigma \in M_p, \sigma' \in M_q$, it is clear that $D_{[\sigma]_{\md_\mv}} \subseteq D_{[\sigma']_{\md_\mv}}$ and thus $f^{\md_\mv}(\sigma) \leq f^{\md_\mv}(\sigma')$.

In subsections [\ref{sec:example_small}, \ref{sec:example_large}] we make use of the downset function to describe persistence of Morse decompositions of multivector fields. Let $(\bar{C}, \bar{d}) := (\bar{C}_{\md_{\mathcal V}},\bar{d}_{\md_{\mathcal V}})$ be a Conley complex induced by the finest Morse decomposition $\md_\mv$ of a multivector field $\mv$ over a simplicial complex $K$. From here on, we take $\md_\mv$ to be minimum and write $f^\mv:=f^{\md_\mv}$. Note that the minimum Morse decomposition of a multivector field $\mv$ may be computed through determining the strongly connected components of an appropriate graph representing $\mv$~\cite{DLMS24}.

The induced \textit{chain function} (see Definition~\ref{def:chainf}) $f^\mv_{\bar{C}}$ assigns to each chain $c \in\bar{C}$ the maximal number of simplices reachable from a single simplex in the support of $c$. This captures, in a sense, the quantity of flow outward from a chain. Cycles with high persistence are boundaries of important sources of flow within the field. Throughout both examples, we consider the continuous vector field shown in Figure~\ref{fig:field_sample}.

\begin{figure}[htbp]
\centerline{\includegraphics[width=.85\textwidth, keepaspectratio]{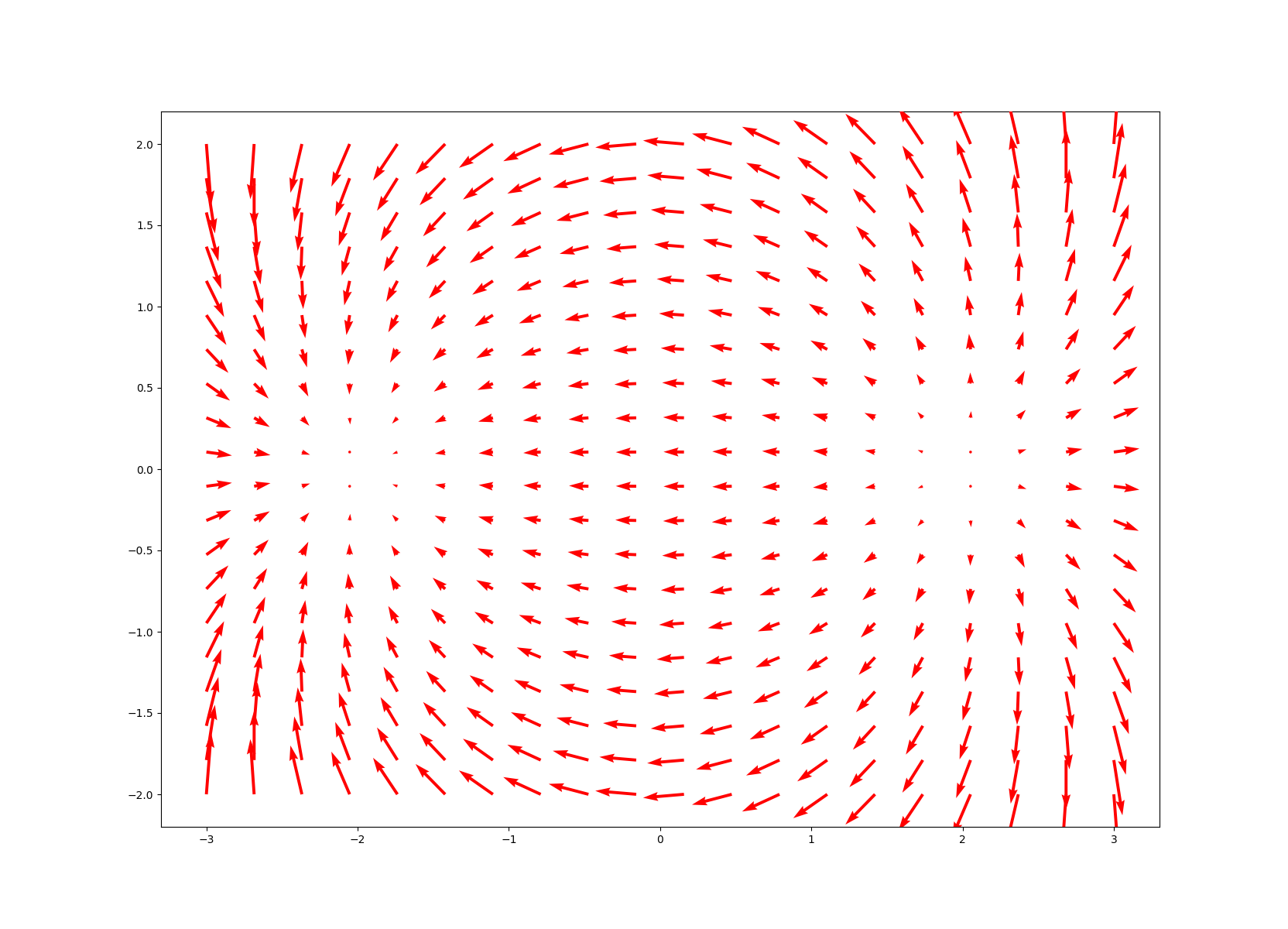}}
\caption{A discrete sample of the vector field determined by the function $g: \mathbb{R}^2 \to \mathbb{R}^2$ where $g(\langle x, y \rangle) = \langle x^2 - y^2 - 4, 2xy \rangle$ over the region $[-3, 3] \times [-2, 2] \subset \mathbb{R}^2$. Observe two critical points: a repeller to the right of the origin at $\langle 2, 0 \rangle$, and an attractor to the left at $\langle -2, 0 \rangle$.}
\label{fig:field_sample}
\end{figure}

\subsubsection{Small example} \label{sec:example_small}

\begin{figure}[htbp]
\centerline{\includegraphics[width=.585\textwidth, keepaspectratio]{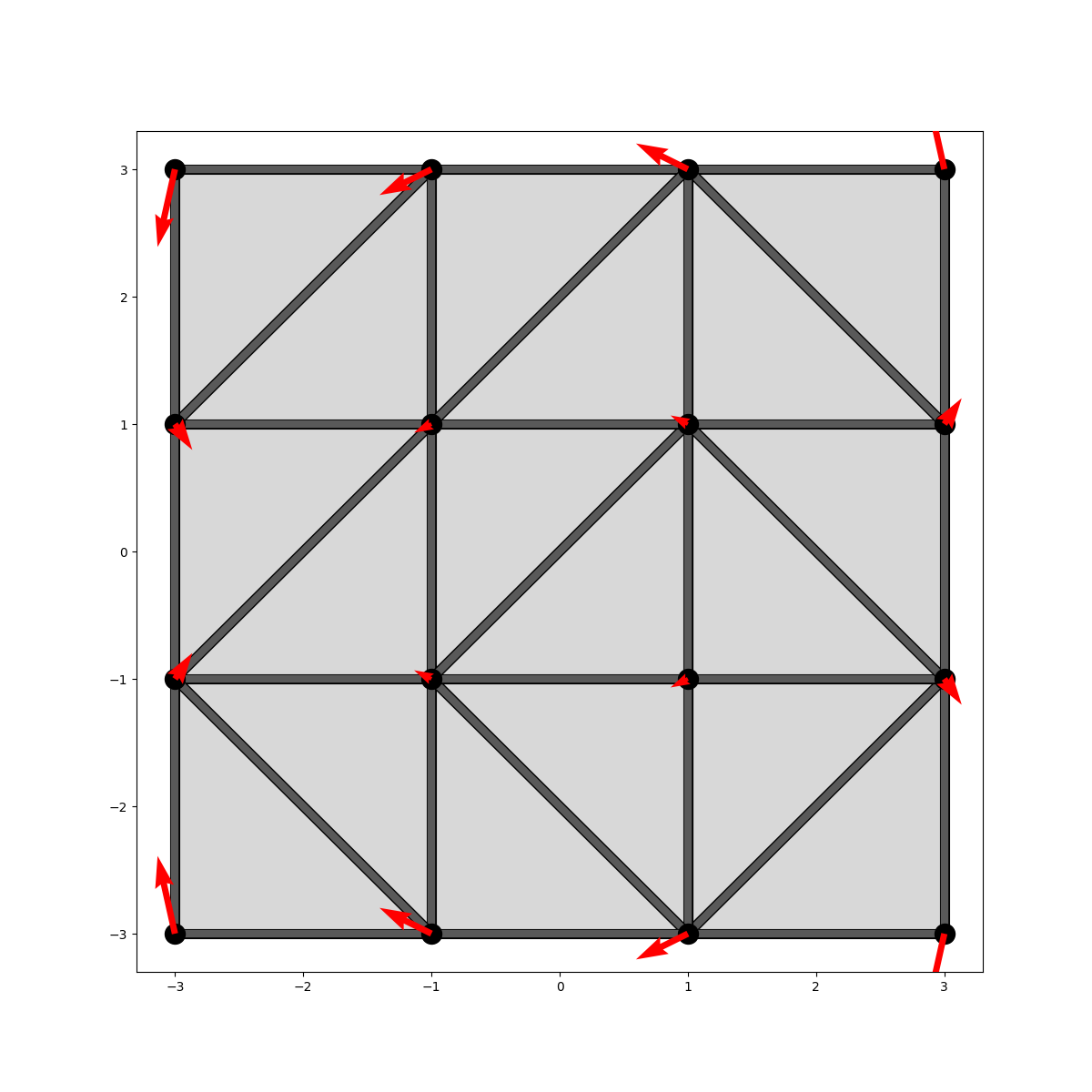}}
\caption{A coarse sample of the field determined by $g$ mentioned in Figure~\ref{fig:field_sample} upon a geometric simplicial complex.}
\label{fig:small_discretization}
\end{figure}

We use a coarse sample of the field $g$ in Figure~\ref{fig:field_sample} over the region $[-3,3]^2 \subset \mathbb{R}^2$ to illustrate the discretization process. We begin by, as is shown in Figure~\ref{fig:small_discretization}, creating a geometric simplicial complex whose vertices coincide with points at which the field is sampled. Multivectors are formed as follows. \textbf{(i)} Initialize the least coarse multivector field over the complex, where each simplex is alone in a singleton vector. \textbf{(ii)} For each vertex with a non-zero sample-vector, determine which coface - if any - of the vertex this sample-vector points into and merge its discrete vector with that of the vertex. \textbf{(iii)} For each edge consider the pseudo sample-vector formed by averaging the sample-vectors at its endpoints, and add the discrete vector of the edge to the discrete vector of the triangle this pseudo sample-vector points into when placed at the edge's midpoint. \textbf{(iv)} Merge the discrete vectors of all vertices and edges along the exterior of the discretization region. Denote this field $\mv$, and the underlying complex $K$.

We compute the minimum Morse decomposition $\md_\mv$ of $\mv$, and the $\md_\mv$-Lyapunov downset function $f^{\mv}$ described in Observation~\ref{obs:downset-lyapunov}. We then apply the procedure described at the end of Section~\ref{sec:persistence}, arriving at a barcode diagram with respect to $f^{\mv}$ of a Conley complex induced by $\md_\mv$, denoted $(\bar{C}, \bar{d}) := (\bar{C}_{\md_{\mathcal V}},\bar{d}_{\md_{\mathcal V}})$. Figure~\ref{fig:small_final} (left) shows the minimum Morse decomposition of $\mv$; simplices which belong to the same discrete vector are assigned the same color. Figure~\ref{fig:small_final} (center) shows a representation of $(\bar{C}, \bar{d})$. The chains corresponding to the columns in $A_{out}$ that are retained in the connection matrix form a basis for $\bar{C}$. Here simplices belonging to the same basis element in $\bar{C}$ are assigned the same color, with one exception: the simplex in $K$ which corresponds to the column of this basis element in the initial input matrix. Notice that this simplex completes (comes last within) the chain basis element. Paths, shown as thin black arrows, represent boundary relations within $(\bar{C}, \bar{d})$; i.e. non-zero entries in the connection matrix. Figure~\ref{fig:small_final} (right) shows the barcode diagram of $(\bar{C}, \bar{d})$ under $f^{\mv}$. The lengths of bars are plotted logarithmically. We observe two short bars corresponding to cycles which become boundaries when the small green and blue chains are introduced, and a long orange bar which corresponds to the cycle which becomes a boundary when the large orange chain is introduced.

\begin{figure}[htbp]
\centerline{\includegraphics[width=1.3\textwidth, keepaspectratio]{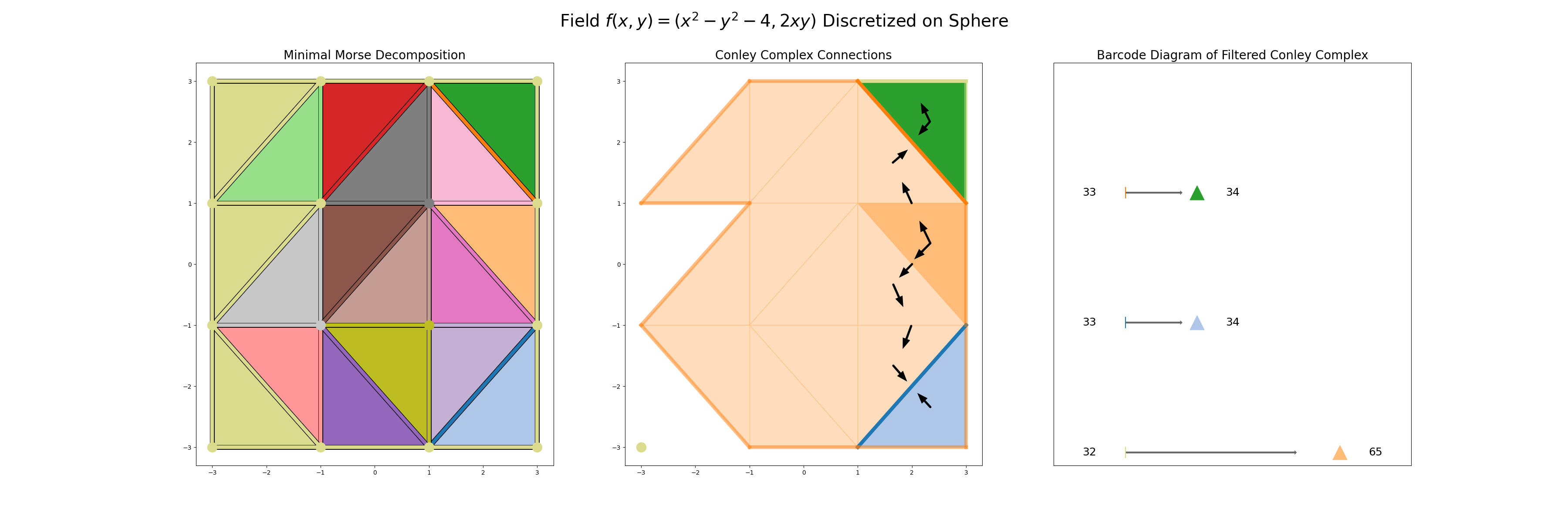}}
\caption{(left) Minimum Morse decomposition of the field $\mv$ - denoted $\md_\mv$, (center) A representation of a Conley complex induced by $\md_\mv$, and (right) the barcode diagram of this Conley complex with respect to the downset function $f^{\mv}$ (Observation~\ref{obs:downset-lyapunov}).}
\label{fig:small_final}
\end{figure}

\subsubsection{Large example} \label{sec:example_large}

We repeat the procedure described in Section~\ref{sec:example_small} with finer sample: our discretization of the field described by function $g$ now consists of a square grid of $21^2$ sample points over the region $[-3,3]^2 \subset \mathbb{R}^2$. Figure~\ref{fig:large_final1} (center) shows a Conley complex dominated by two large repellers, with flows terminating at the attractor described in the caption of Figure~\ref{fig:field_sample}. Figure~\ref{fig:large_final1} (right) shows a persistence diagram with two bars, corresponding to the cycles which encircle these repellers. Observe that both discretizations at coarser and finer levels bring in extra repellers in addition to the original one for the vector field shown in Figure~\ref{fig:field_sample}. However, finer resolution reduces the error by reducing the number of repellers from three to two. At higher resolution, we discern properties of the field more accurately. 

\begin{figure}[htbp]
\centerline{\includegraphics[width=1.3\textwidth, keepaspectratio]{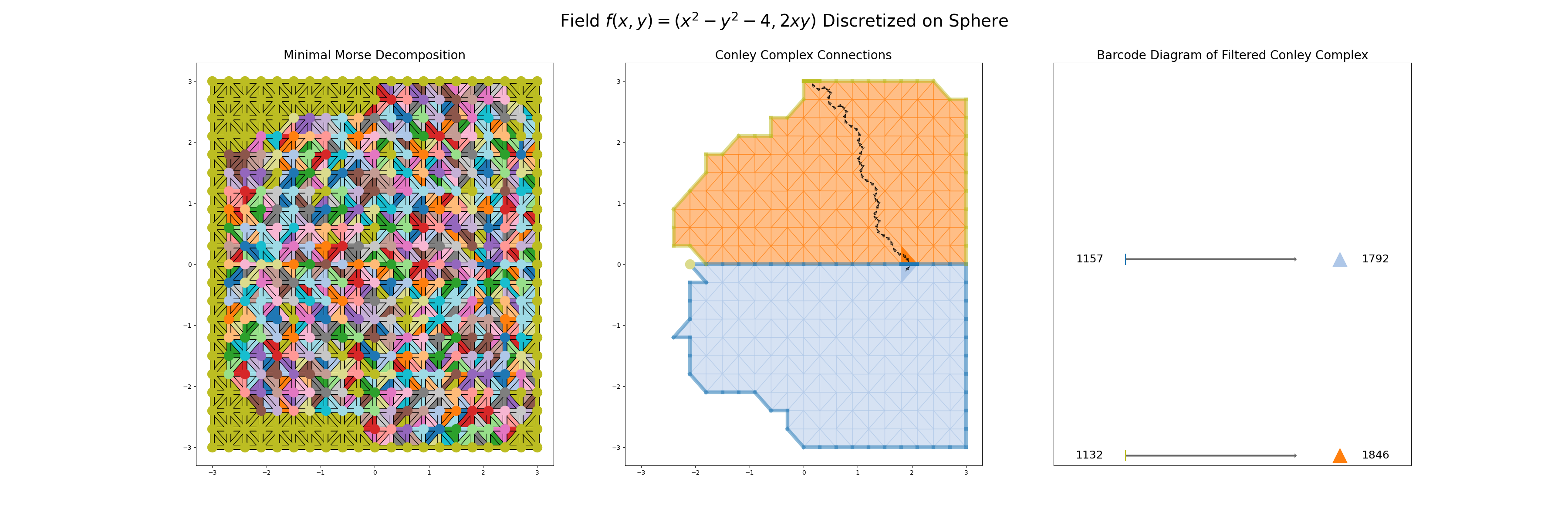}}
\caption{(left) Minimum Morse decomposition of the field $\mv$, (center) A representation of a Conley complex induced by $\mv$, and (right) the barcode diagram of this Conley complex under downset function $f^{\mv}$ (Observation~\ref{obs:downset-lyapunov}).}
\label{fig:large_final1}
\end{figure}

 \section{Conclusions} We have shown that a variation of the well known persistence algorithm with exhaustive reductions produces a connection matrix for a given
    Morse decomposition of a combinatorial (multi)vector field. The implications are twofold: first, it brings
    the theory of combinatorial dynamics closer to the theory of persistence; second,
    it allows a more efficient way of computing connection matrices compared
    to the best known algorithm in~\cite{DLMS24} - see our experimental results in section~\ref{sec:benchmark}. Because of the restrictions of homogeneity on column additions, well-known optimization techniques~\cite{bauer2014clear,wagner2012} for persistence algorithms could not be used
    in our experiments. Future work can be directed to address this shortcoming. It is also worth pursuing further application of the persistence of Lyapunov functions on Morse decompositions of (multi)vector fields. For example, it may aid in classifications of (discrete)Morse functions and the flows which they induce.

    \section*{Acknowledgment} This research is supported by the NSF grants DMS-2301360 and CCF-2437030.
    This project has received funding from the European Union’s Horizon 2020 research and innovation programme under the Marie Skłodowska-Curie Grant Agreement No. 101034413.
    Also, we acknowledge many helpful comments and discussions with Marian Mrozek.

\bibliography{refs}
%}
%\section{Missing Proofs}
%{\bf Proof of Proposition~\ref{prop:matrices}}.\\

%\noindent                                      

\end{document}